\documentclass{article}
\usepackage{listings}
\usepackage{amsmath, amsthm, amssymb}
\newtheorem{prop}{Proposition}
\newtheorem{ex}{Example}

\newtheorem{lemma}{Lemma}
\newtheorem{remark}{Remark}
\newtheorem{teo}{Theorem}
\newtheorem{definition}{Definition}
\usepackage{graphics}
\usepackage{graphicx}
\usepackage[scale=0.75]{geometry}
\usepackage[utf8]{inputenc}
\usepackage[english]{babel}

\usepackage{listings}
\usepackage{url}
\newcommand\blfootnote[1]{%
	\begingroup
	\renewcommand\thefootnote{}\footnote{#1}%
	\addtocounter{footnote}{-1}%
	\endgroup
}

\author{
	Roberta Flenghi$^1$\\
	\texttt{roberta.flenghi@enpc.fr}
	\and
	Benjamin Jourdain$^1$\\
	\texttt{benjamin.jourdain@enpc.fr}
}

\date{%
	$^1$Cermics, \'Ecole des Ponts, INRIA, Marne-la-Vall\'ee, France.\\%

}
\graphicspath{ {./images/} }
\usepackage[output-decimal-marker={,}]{siunitx}
\usepackage{booktabs}
\usepackage{enumitem}

\usepackage{relsize}

\usepackage{xcolor}
\usepackage{tcolorbox}
\usepackage{verbatim}
\usepackage{bm}

\newcounter{hypoconbisfl}
\newcounter{saveconbisfl}
\newcommand\debutL{\begin{list} {\textbf{L\arabic{hypoconbisfl}}}{\usecounter{hypoconbisfl}}\setcounter{hypoconbisfl}{\value{saveconbisfl}}}
	\newcommand\finL{\end{list}\setcounter{saveconbisfl}{\value{hypoconbisfl}}}

\newcounter{hypoconbisf}
\newcounter{saveconbisf}
\newcommand\debutRU{\begin{list} {\textbf{RU\arabic{hypoconbisf}}}{\usecounter{hypoconbisf}}\setcounter{hypoconbisf}{\value{saveconbisf}}}
	\newcommand\finRU{\end{list}\setcounter{saveconbisf}{\value{hypoconbisf}}}

\usepackage{comment}
\newcounter{hypoconbisx}
\newcounter{saveconbisx}
\newcommand\debutTX{\begin{list} {\textbf{TX\arabic{hypoconbisx}}}{\usecounter{hypoconbisx}}\setcounter{hypoconbisx}{\value{saveconbisx}}}
	\newcommand\finTX{\end{list}\setcounter{saveconbisx}{\value{hypoconbisx}}}

\makeatletter
\newcommand{\mathleft}{\@fleqntrue\@mathmargin0pt}
\newcommand{\mathcenter}{\@fleqnfalse}
\makeatother
\title{Central limit theorem over non-linear functionals of empirical measures: beyond the iid setting}

\begin{document}

	\maketitle	
	\blfootnote{``This work is
	supported by the french National Research Agency under the grant
	ANR-21-CE40-0006 (SINEQ).'' 
}
	\section*{Introduction}
	In this work we are interested in the convergence in distribution of
	
	\begin{equation}\label{purpose}
		\sqrt{N}\left(U\left(\frac{1}{N}\sum_{i=1}^{N}\delta_{X_i} \right)-U\left(\mu\right)\right)  
	\end{equation}
	
	where $U$ is a general function defined on some Wasserstein space of probability measures on $\mathbb{R}^d$ and $\left( X_i\right)_{i\geq 1} $ is a given sequence of $\mathbb{R}^d$-valued random vectors.\\ In the mathematical statistics literature, Von Mises \cite{vonmises1} \cite{vonmises2} (see also Chapter 6 of \cite{VonMises}) was the first to develop  an approach for deriving the asymptotic distribution theory of functionals $U(\frac{1}{N}\sum_{i=1}^{N}\delta_{X_i})$ (called ``statistical functions'') under the assumption of independent and identically distributed real-valued random variables $\left( X_i\right)_{i\geq 1} $. \\Through the use of the G\^{a}teaux differential, he introduced a Taylor expansion of $U(\frac{1}{N}\sum_{i=1}^{N}\delta_{X_i})$ around $U(\mu)$ where $\mu$ is the common distribution of the random variables. He proved that if the linear term is the first 
	nonvanishing term in the Taylor expansion of the functional $U(.)$ at $\mu$, the limit distribution 
	is normal (under the usual restrictions corresponding to the central limit 
	theorem). One of the main difficulty was to prove that the remainder in the Taylor expansion goes to zero.\\
	In dimension $d=1$, Boos and Serfling \cite{BoosSerfling} assumed the existence of a differential for $U(.)$ at $\mu$ in a sense stronger than the G\^{a}teaux differential. They assumed the existence of $\frac{d}{d\epsilon}_{|\epsilon=0^+}U\left(\mu+\epsilon\left(\nu-\mu\right)  \right)= dU\left(\mu,\nu-\mu\right)$ at $\mu$ that is linear in $\nu-\mu$ (for $\nu$ any probability measure on the real line ) and such that  
	\begin{equation} \label{boosserfling}
		U\left(\frac{1}{N}\sum_{i=1}^{N}\delta_{X_i} \right)-U\left(\mu\right)=\frac{1}{N}\sum_{i=1}^{N}dU\left(\mu,\delta_{X_i}-\mu \right)+ o\left(\left\|\frac{1}{N}\sum_{i=1}^{N}1_{\left\lbrace X_i\leq\cdot \right\rbrace }-\mu\left(-\infty,\cdot \right]  \right\|_\infty  \right) .
	\end{equation}
	
	Taking 
	advantage of known stochastic properties of the Kolmogorov-Smirnov distance, in particular the boundness in probability of $\left(\sqrt{N}\left(\left\|\frac{1}{N}\sum_{i=1}^{N}1_{\left\lbrace X_i\leq\cdot \right\rbrace }-\mu\left(-\infty,\cdot \right]  \right\|_\infty  \right)  \right)_{N\geq1}$, they conclude that $\sqrt{N}\left( U\left(\frac{1}{N}\sum_{i=1}^{N}\delta_{X_i} \right)-U\left(\mu\right)\right) $ converges in distribution to a centered Gaussian random
	variable with asymptotic variance equal to the common variance of the independent and identically	distributed random variables $dU\left(\mu,\delta_{X_i}-\mu \right)$ when they are square integrable and centered. In addition to the limitation of their approach to dimension $d=1$, it relies on the uniformity in \eqref{boosserfling} of the approximation with respect to the Kolmogorov-Smirnov distance which is a strong assumption almost amounting to Fréchet differentiability of $U$ at $\mu$ for the Kolmogorov norm.
	When $\mu$ is a probability measure on any measurable space, Dudley \cite{Dudley} obtains central limit theorems for $\sqrt{N}(U(\frac{1}{N}\sum_{i=1}^N\delta_{X_i})-U(\mu))$ under Fréchet differentiability of $U$ at $\mu$ with respect to $\|\nu-\mu\|=\sup_{f\in{\mathcal F}}\left|\int f(x)(\nu-\mu)(dx)\right|$ where the class ${\mathcal F}$ of measurable functions is such that a central limit theorem for empirical measures holds with respect to uniform convergence over ${\mathcal F}$. Clearly the requirements on ${\mathcal F}$ impose some balance: the larger the class ${\mathcal F}$, the easier Fréchet differentiability becomes, but the stronger the uniform convergence over ${\mathcal F}$ becomes. \\
	Recently Jourdain and Tse \cite{Jourdain} reconsidered the same problem under the assumption of independent and identically distributed $\mathbb{R}^d$-valued random vectors $\left(X_i \right)_{i\geq1} $. By means of the notion of the linear functional derivative of $U$ that is a G\^{a}teaux differential with the property that $dU\left(\mu,\nu-\mu\right)=\int_{\mathbb{R}^d}\frac{\delta U}{\delta m}\left(\mu,y \right)\left(\nu-\mu\right)(dy) $ for some measurable real valued function $\mathbb{R}^d\ni y\rightarrow \frac{\delta U}{\delta m}\left(\mu,y \right)$
	with some polynomial growth assumption in $y$,  they linearize $\sqrt{N}\left(U\left(\frac{1}{N}\sum_{i=1}^{N}\delta_{X_i} \right)-U\left(\mu\right)\right)$ into the sum of

	\begin{equation} \label{linearization}
		\frac{1}{\sqrt{N}}\sum_{i=1}^{N}\left( \frac{\delta U}{\delta m}\left(\frac{1}{N}\sum_{j=1}^{i-1}\delta_{X_j}+\frac{N+1-i}{N}\mu,X_i \right)- \int_{\mathbb{R}^d} \frac{\delta U}{\delta m}\left(\frac{1}{N}\sum_{j=1}^{i-1}\delta_{X_j}+\frac{N+1-i}{N}\mu,x \right)\mu(dx)\right)
	\end{equation} 
	and a remainder. Such a decomposition allows to apply to the above sum the Central Limit Theorem for arrays of martingale increments and to investigate sufficient conditions for the remainder to vanish in probability. They finally proved that 
	
	$$\sqrt{N}\left(U\left(\frac{1}{N}\sum_{i=1}^{N}\delta_{X_i} \right)-U\left(\mu\right)\right) \overset{d}{\Longrightarrow } \mathcal{N}\left(0, Var\left( \frac{\delta U}{\delta m}\left(\mu,X_1 \right)\right)  \right).$$
        They replace the uniformity leading to Fréchet differentiability required in the statistical literature, by supposing that the linear functional derivative exists not only at $\mu$ but on a Wasserstein ball with positive radius containing $\mu$. This is a mild restriction, since when a central limit theorem holds for some statistical functional, it is in general not limited to a single value of the common distribution $\mu$ of the samples.

        The aim of this work is to generalize what has been done by Jourdain and Tse. We will first relax the equidistribution assumption by studying the convergence of (\ref{purpose}) by assuming that the $\mathbb{R}^d$-valued random vectors $(X_i)_{i\geq1}$ are independent and non-equidistributed and we denote the law of $X_i$ by $\nu_{i}$.
	Since in our case we do not assume the equidistribution of the random variables, we need to give sufficient conditions for  $\frac{1}{N}\sum_{i=1}^{N}\delta_{X_i}$ and for $\frac{1}{N}\sum_{i=1}^{N}\nu_{i} $ to converge to a common limit $\mu$ (for a distance that will be specified later). We will split $ \sqrt{N}\left(U\left(\frac{1}{N}\sum_{i=1}^{N}\delta_{X_i} \right)-U\left(\mu\right)\right)$ into the sum
	\begin{equation}
		\sqrt{N}\left(U\left(\frac{1}{N}\sum_{i=1}^{N}\nu_{i} \right)-U\left(\mu\right)\right)+\sqrt{N}\left(U\left(\frac{1}{N}\sum_{i=1}^{N}\delta_{X_i} \right)-U\left(\frac{1}{N}\sum_{i=1}^{N}\nu_{i}\right)\right) .
	\end{equation} \\
	We will give sufficient conditions for the first component to converge to a constant that will be specified later. For what concerns the second component we generalize the linearization argument by means of the linear functional derivative $\frac{\delta U}{\delta m}$ in (\ref{linearization}) to prove the convergence in distribution to a gaussian random variable.\newline \\
	We next get rid of the  independence hypothesis by assuming that the sequence of random vectors $\left( X_i\right)_{i\geq 1} $ to be a $\mathbb{R}^d$-valued Markov chain with transition kernel $P$ and unique invariant probability measure $\mu$.
	By using the linear functional derivative $\frac{\delta U}{\delta m}$, we again linearize $\sqrt{N}\left(U\left(\frac{1}{N}\sum_{i=1}^{N}\delta_{X_i} \right)-U\left(\mu\right)\right)$ into the sum of $(\ref{linearization})$
	and a remainder. Under assumptions on the Markov kernel $P$ that ensure that the central limit theorem holds for linear functionals and by giving sufficient conditions for the remainder to vanish in probability as $N$ goes to infinity, we conclude that 
	$$\sqrt{N}\left(U\left(\frac{1}{N}\sum_{i=1}^{N}\delta_{X_i} \right)-U\left(\mu\right)\right) \overset{d}{\Longrightarrow } \mathcal{N}\left(0, \mu\left(PF^2\left(\mu, \cdot \right)\right)-\mu\left( \left(PF \right)^2\left(\mu, \cdot \right)  \right)\right)$$ 
	
	with $F$ solution of the Poisson equation 
	$$F\left(\mu,x \right)-PF\left(\mu,x \right)= \frac{\delta U}{\delta m}\left(\mu,x\right)- \int_{\mathbb{R}^d}\frac{\delta U}{\delta m}\left(\mu,y\right)\mu(dy), \quad x\in \mathbb{R}^d.$$ 
	As far as we know, such a generalization to non i.i.d. random vectors $X_i$ appears for the first time in the literature. This illustrates that the linear functional derivative is a versatile tool. \\
	
	In the first section, we will provide the statement of the two results. Each of them will be preceeded by reminders useful for its understanding. Together with the independent non-equidistributed case we will recall the notions of Wasserstein distance and linear functional derivative and together with the Markov chains case we will recall definitions and general facts about Markov chains and the Poisson equation.
	In the second section, the proofs of the two results are given. 
	
	\section{Main Results}
	\subsection{Independent Non-Equidistributed Random Variables} 
	Let us observe that the case of a linear functional $U$, that is $U(m)=\int_{\mathbb{R}^d}f(x)m(dx)$ with $f:\mathbb{R}^d\rightarrow \mathbb{R}$ a measurable function, has been largely studied in this context.\\
	We recall that a sequence $(f(X_i))_{i\geq 1}$ is Strongly Residually Cesaro $\beta$-Integrable for some $\beta >0$ (SRCI($\beta$), in short) if 
	\begin{enumerate}[label=(\roman*)]
		\item $\sup_{N\geq 1}\dfrac{1}{N}\mathlarger{\sum}_{i=1}^{N}\mathbb{E}(|f(X_i)|)<\infty$
		\item $\mathlarger{\sum}_{i=1}^{\infty}\dfrac{1}{i}\mathbb{E}\left(\left( |f(X_i)|-i^\beta\right) 1_{\left\lbrace|f(X_i)|> i^\beta \right\rbrace } \right)  < \infty $.
	\end{enumerate}
	
	Chandra and Goswami \cite[Theorem 4.1]{chandra} proved that if $(f(X_i))_{i\geq 1}$ is a sequence of random variables  pairwise independent and verifying the condition SRCI($\beta$) for some $\beta \in (0,1)$, then the Strong Law of Large Numbers holds:
	
	\begin{equation} \label{lln}
		\lim_{N\rightarrow \infty} \dfrac{\sum_{i=1}^{N} (f(X_i)-\mathbb{E}(f(X_i)))}{N} = 0 \quad a.s. 
	\end{equation} 
	
	Moreover Lindeberg proved (see for instance \cite{billingsley}) that, also in this case, the Central Limit Theorem holds. More in detail, consider a sequence $(f(X_i))_{i\geq 1}$  of square-integrable independent random variables such that $\lim_{N\rightarrow \infty}\frac{\mathlarger{\sum}_{i=1}^{N}Var(f(X_i))}{N}=\sigma^2$ where $\sigma^2>0$. If moreover the Lindeberg's condition is satified
	$$ \forall \epsilon >0 \quad \lim_{N\rightarrow \infty}\dfrac{1}{N} \mathlarger{\sum}_{i=1}^{N} \mathbb{E}((f(X_i)-\mathbb{E}(f(X_i)))^21_{\left| f(X_i)-\mathbb{E}(f(X_i))\right|>\epsilon \sqrt{N} })=0, $$
	then 
	
	\begin{equation} \label{clt}
		\dfrac{\sum_{i=1}^{N} (f(X_i)-\mathbb{E}(f(X_i)))}{\sqrt{N}} \overset{d}{\Longrightarrow } \mathcal{N}(0,\sigma^2).
	\end{equation} 
	Before giving the statement of the Central Limit Theorem for general $U$ in the case of independent non-equidistributed random variables, let us recall some notions about the Wasserstein distance and the linear functional derivative.
	\subsubsection{The Wasserstein distance and the linear functional derivative} 
	
	Let $U:\mathcal{P}_\ell\left( \mathbb{R}^d\right)\rightarrow \mathbb{R}$ where for $\ell \geq 0$ we denote by $\mathcal{P}_\ell\left( \mathbb{R}^d\right)$ the set of probability measures $m$ on $\mathbb{R}^d$ such that $\int_{\mathbb{R}^d}\left|x \right|^\ell m(dx)<\infty$. For $\ell>0$, we consider the $\ell$-Wasserstein metric defined for $\mu_1,\mu_2 \in \mathcal{P}_\ell\left( \mathbb{R}^d\right)$ by 
	\begin{equation} \label{wasserstein}
		W_\ell\left(\mu_1,\mu_2\right) = \inf\left\lbrace \left(\int_{\mathbb{R}^{2d}}\left|x-y \right|^\ell \rho\left(dx,dy\right)  \right)^{\frac{1}{\ell \vee 1}} : \rho\in \mathcal{P}(\mathbb{R}^{2d})\,with\,  \rho\left(\cdot\times\mathbb{R}^d \right)=\mu_1(\cdot), \rho\left(\mathbb{R}^d\times \cdot \right)=\mu_2(\cdot)  \right\rbrace.	
	\end{equation}
	If $\mu \in \mathcal{P}_\ell\left(\mathbb{R}^d\right)$ and $\left(\mu_n\right)_{n\in \mathbb{N}}$ is a sequence in this space, then $\lim_{n\to\infty} W_\ell\left(\mu_n,\mu \right)=0$ if and only if $\lim_{n\rightarrow \infty}\int_{\mathbb{R}^d}|x|^\ell\mu_{n}(dx)=\int_{\mathbb{R}^d}|x|^\ell\mu(dx)$ and $\mu_n$ converges weakly to $\mu$ as $n\to\infty$ where we write $\mu_n\rightharpoonup  \mu$ to denote the weak convergence. Alternatively $\lim_{n\to\infty} W_\ell\left(\mu_n,\mu \right)=0$ if and only if $\forall \phi:\mathbb{R}^d\rightarrow \mathbb{R}$ continous $\mu$-almost everywhere and such that $\sup_{x\in \mathbb{R}^d}\frac{|\phi(x)|}{1+|x|^\ell}<\infty,$
	\begin{equation} \label{carat_wass}
		\lim_{n\to\infty}\int_{\mathbb{R}^d}\phi(x)\mu_{n}(dx)=\int_{\mathbb{R}^d}\phi(x)\mu(dx).
	\end{equation}\\
	For $\ell=0$ and $\mu_1,\mu_2 \in \mathcal{P}_0\left( \mathbb{R}^d\right)$, we consider 
	
	$$W_0\left(\mu_1,\mu_2\right) = \inf\left\lbrace \int_{\mathbb{R}^{2d}}(1\wedge\left|x-y\right|)\rho\left(dx,dy\right)   :\rho\in \mathcal{P}(\mathbb{R}^{2d})\,with\,  \rho\left(\cdot\times\mathbb{R}^d \right)=\mu_1(\cdot), \rho\left(\mathbb{R}^d\times \cdot \right)=\mu_2(\cdot)   \right\rbrace.$$
	Notice that $W_0$ metricizes the topology of weak convergence.
%	Notice that 
%	\begin{align*}
%		W_0\left( \mu_1,\mu_2 \right)&=\sup_{A\in \mathcal{B}\left(\mathbb{R}^d \right) }\left|\mu_1(A)-\mu_2(A) \right|=\frac{1}{2}\left|\mu_1-\mu_2 \right|\left(\mathbb{R}^d \right)=d_{TV}(\mu_1,\mu_2)\\
%		&=\frac{1}{2}\sup_{\left\|f \right\|_\infty\leq 1 }\left|\int_{\mathbb{R}^d}f(x)(\mu_1(dx)-\mu_2(dx)) \right|,
%	\end{align*}
%	where $\mathcal{B}\left(\mathbb{R}^d \right)$ denotes the Borel $\sigma$-algebra of $\mathbb{R}^d$, $\left|\mu_1-\mu_2 \right|$ the total variation of the signed measure $\mu_1-\mu_2$ and $d_{TV}$ the total variation distance between $\mu_1$ and $\mu_2$.\\ 
	For $\ell\geq 0$ we can also consider $\mathcal{M}_\ell(\mathbb{R}^d)$, the space of the signed measures $\tau$ on $\mathbb{R}^d$ such that $\int_{\mathbb{R}^d}|x|^\ell |\tau|(dx)<\infty$ where $|\cdot|$ denotes the total variation of a signed measure.
	For each $\tau \in \mathcal{M}_\ell(\mathbb{R}^d)$ we will define the norm $\left\|\tau \right\|_\ell= \sup_{f:|f(x)|\leq 1+|x|^\ell} \int_{\mathbb{R}^d}f(x)\tau(dx)$ where the supremum is computed over the set of the measurable functions satisfying the growth condition and it can be proved that given $\tau\in\mathcal{M}_\ell(\mathbb{R}^d)$ and $(\tau_n)_{n\in \mathbb{N}}$ a sequence in this space such that $\left\|\tau_n-\tau \right\|_\ell\rightarrow 0,$ then $\left\||\tau_n|-|\tau| \right\|_\ell\rightarrow 0.$\\
	Let us observe that in $\mathcal{P}(\mathbb{R}^d)$ the convergence with respect to $\left\|\cdot \right\|_\ell$ implies the convergence with respect to $W_\ell$. Moreover if $\ell=0$, for $\mu_1,\mu_2\in \mathcal{P}(\mathbb{R}^d)$ we have
	\begin{align*}
				\left\|\mu_1-\mu_2 \right\|_0&=\sup_{\left\|f \right\|_\infty\leq 1 }\left|\int_{\mathbb{R}^d}f(x)(\mu_1(dx)-\mu_2(dx)) \right| = 2d_{TV}(\mu_1,\mu_2)\\
				&=2\sup_{A\in \mathcal{B}\left(\mathbb{R}^d \right) }\left|\mu_1(A)-\mu_2(A) \right|=\left|\mu_1-\mu_2 \right|\left(\mathbb{R}^d \right)
	\end{align*}
	where $\mathcal{B}\left(\mathbb{R}^d \right)$ denotes the Borel $\sigma$-algebra of $\mathbb{R}^d$ and $d_{TV}$ the total variation distance between $\mu_1$ and $\mu_2$.
	Let us now recall the notion of (first order) linear functional derivative associated to $U$. For a more detailed description, including the definition of the linear functional derivative of a superior order, see \cite{linearfunctionalderivative}.

	\begin{definition}
		Let $\ell \geq 0$. A function $U:\mathcal{P}_\ell\left(\mathbb{R}^d\right) \rightarrow \mathbb{R}$ admits a linear functional derivative at $\mu \in \mathcal{P}_\ell\left(\mathbb{R}^d\right)$ if there exists a measurable function $\mathbb{R}^d\ni y \mapsto \frac{\delta U}{\delta m}\left(\mu,y\right) $ such that $\sup_{y\in \mathbb{R}^d}\frac{\left| \frac{\delta U}{\delta m}\left(\mu,y\right)\right| }{1+|y|^\ell}< \infty$ and
		$$\forall \nu \in \mathcal{P}_\ell\left(\mathbb{R}^d\right), \frac{d}{d\epsilon}_{\lvert_{\epsilon=0^+}} U\left(\mu+\epsilon(\nu-\mu) \right)= \int_{\mathbb{R}^d}  \frac{\delta U}{\delta m}\left(\mu,y\right) (\nu - \mu)(dy).$$ 
	\end{definition}
	
	The next lemma allows to express a finite difference of the function $U$ as an integral of the functional linear derivative. 
	
	\begin{lemma} \label{linearderivative}
		Let $\ell \geq 0$, $m$, $m'\in \mathcal{P}_\ell\left(\mathbb{R}^d\right)$ and suppose that the linear functional derivative of a function $U:\mathcal{P}_\ell\left(\mathbb{R}^d\right)  \rightarrow \mathbb{R} $ exists in the segment $\left(m_s:= sm'+(1-s)m\right)_{s\in \left[ 0,1\right] } $. Then if $\sup_{(s,y)\in \left[0,1 \right]\times \mathbb{R}^d}\frac{\left| \frac{\delta U}{\delta m}\left(m_s,y\right)\right| }{1+|y|^\ell}< \infty$, one has
		\begin{equation}\label{linearfunc}
			U(m')-U(m)= \int_{0}^{1}\int_{\mathbb{R}^d}\frac{\delta U}{\delta m}((1-s)m+sm',y)(m'-m)(dy)ds.
		\end{equation}
		
	\end{lemma}

	\subsubsection{Statement of the theorem}

	Given a measure $\mu\in\mathcal{P}_\ell\left(\mathbb{R}^d \right)$, we will consider the following group of hypotheses about the regularity of the functional derivative $U$ (\textbf{RU}) in a neighborhood of $\mu$:
	
	\debutRU
	\item \label{functional_derivative2} there exists $r>0$ such that $U$ admits a linear functional derivative on the ball $B(\mu,r)$ centered at $\mu$ with radius $r$ for the metric $W_\ell$
	\item \label{growth2} $\exists C < \infty, \forall (\tilde{\mu},x) \in B(\mu,r) \times \mathbb{R}^d, \left| \dfrac{\delta U}{\delta m}(\tilde{\mu},x)\right| \leq C \left( 1 + |x|^{\frac{\ell}{2}}\right) $
	\item\label{derivative_conv2} $ \sup_{x\in \mathbb{R}^d} \dfrac{|\frac{\delta U}{\delta m}(\tilde{\mu},x)-\frac{\delta U}{\delta m}(\mu,x)|}{1+|x|^\frac{\ell}{2}} $ converges to $0$ when $W_\ell(\tilde{\mu},\mu)$ goes to $0$
	\item \label{continuity} $x\mapsto \frac{\delta U}{\delta m}\left(\mu,x \right) $ is continuous $\mu$-almost everywhere
	\item \label{growth_diff2} $\exists \alpha\in (\frac{1}{2},1], \exists C < \infty, \forall \mu_1,\mu_2\in B(\mu,r), \forall x\in \mathbb{R}^d$
	$$\left| \dfrac{\delta U}{\delta m}(\mu_2,x)-\dfrac{\delta U}{\delta m}(\mu_1,x)\right| \leq C\left( (1+|x|^\ell)\left\| \mu_2-\mu_1\right\|_0^\alpha +(1+|x|^{\ell (1-\alpha)})\left( \int_{\mathbb{R}^d}|y|^\ell |\mu_2-\mu_1|(dy) \right)^\alpha \right).  $$
	\finRU

	Moreover we will consider the following assumption about the tails of the random vectors $(X_i)_{i\geq1}$:
\begin{itemize}
		\item[\textbf{TX}]\label{srci}there exists $\beta \in (0,1)$ such that $\mathlarger{\sum}_{i=1}^{\infty}\dfrac{1}{i}\mathbb{E}\left(\left( |X_i|^\ell-i^\beta\right) 1_{\left\lbrace|X_i|^\ell> i^\beta \right\rbrace } \right)  < \infty. $
\end{itemize}
			
	Let us observe that it coincides with condition $(ii)$ of the Strongly Residually Cesaro $\beta-$Integrability with the choice of $f(x)= x^\ell$.
	We are now ready to state respectively the Strong Law of Large Numbers and the Central Limit Theorem.
	
	\begin{teo}[LLN for independent non-equidistributed
		r.v.]\label{SLLN}
		Let $\ell \geq 0$ and let $X_i,\, i\geq 1$ be a sequence of independent random variables on $\mathbb{R}^d$ with law $\nu_i \in \mathcal{P}_{\ell}(\mathbb{R}^d)$ and let us define 
		$$\mu_N :=\frac{1}{N}\sum_{i=1}^{N} \delta_{X_i}$$
		and 
		$$\bar{\nu}_N:= \mathbb{E}\left(  \mu_N\right)  = \frac{1}{N}\sum_{i=1}^{N} \nu_i.$$
		Let us assume \textbf{TX} and the existence of $\mu \in \mathcal{P}_\ell(\mathbb{R}^d)$ such that $\lim_{N\rightarrow \infty}W_\ell(\bar{\nu}_N,\mu)=0.$ Then 
		$$W_\ell\left(\mu_N, \mu\right) \underset{N}{\longrightarrow} 0 \quad a.s.$$	
	\end{teo}
	\begin{proof} Thanks to the existence of $\mu \in \mathcal{P}_\ell(\mathbb{R}^d)$ such that $\lim_{N\rightarrow \infty}W_\ell(\bar{\nu}_N,\mu)=0$, by the characterization  of the Wasserstein convergence, one has $ \bar{\nu}_N \rightharpoonup \mu$ and so by Wellner paper \cite{Wellner}  $\mu_N \rightharpoonup \mu$ a.s..\\
		If $\ell=0$ the proof is completed, let us therefore suppose that $\ell>0$ and let us check the convergence of the $\ell$-th moment.
		Since the $\left| X_i\right|^\ell$'s are independent,  Assumption \textbf{TX} holds and  $\sup_{N\geq 1}\frac{1}{N}\sum_{i=1}^{N}\mathbb{E}\left(\left| X_i\right|^\ell \right)=\sup_{N\geq 1} \int_{\mathbb{R}^d}|x|^\ell\bar{\nu}_N(dx) <\infty$ (using again the characterisation of the Wasserstein convergence), we can apply (\ref{lln}) and obtain that	
		
		$$\int_{\mathbb{R}^d} | x |^\ell \mu_N\left(dx\right) -  \int_{\mathbb{R}^d} | x |^\ell \bar{\nu}_N\left(dx\right)=\dfrac{| X_1|^\ell+\cdots+| X_N|^\ell - \left( \mathbb{E}\left( | X_1|^\ell\right) +\cdots+ \mathbb{E}\left( | X_N|^\ell\right)  \right) }{N} \underset{N\rightarrow \infty}{\longrightarrow} 0 \qquad a.s. $$  
		
		Therefore
		
		\begin{align*}
			\lim_{N\rightarrow \infty}\int_{\mathbb{R}^d}& | x |^\ell \mu_N\left(dx\right) -  \int_{\mathbb{R}^d} | x |^\ell \mu\left(dx\right)\\
			&=\lim_{N\rightarrow \infty} \int_{\mathbb{R}^d} | x |^\ell \mu_N\left(dx\right) -  \int_{\mathbb{R}^d} | x |^\ell \bar{\nu}_N\left(dx\right)+  \int_{\mathbb{R}^d} | x |^\ell \bar{\nu}_N\left(dx\right) - \int_{\mathbb{R}^d} | x |^\ell \mu\left(dx\right) = 0 \qquad a.s.
		\end{align*}

	\end{proof} 
	\begin{teo}[CLT for independent non-equidistributed
		r.v.]\label{independent_case}
		Using the same notations of Theorem \ref{SLLN}, let us assume that

		\begin{enumerate}		
			\item \label{wess.conv2} $W_\ell\left(\frac{1}{N}\sum_{i=1}^N\nu_i(dx)\nu_i(dy),\eta(dx,dy)\right)  \underset{N\rightarrow \infty}{\longrightarrow} 0$ for some measure $\eta \in \mathcal{P}_\ell (\mathbb{R}^d \times \mathbb{R}^d)$ with $\mu(\cdot) = \eta(\mathbb{R}^d,\cdot)=\eta(\cdot,\mathbb{R}^d)$
			
			\item\label{l_convergence}
			$\|\sqrt{N}\left(\bar{\nu}_N-\mu \right) - \sigma\|_\ell \underset{N\rightarrow \infty}{\longrightarrow} 0$
			for some measure $\sigma$ in $\mathcal{P}_\ell(\mathbb{R}^d). $
		\end{enumerate}
		
		If moreover \textbf{RU\ref{functional_derivative2}-\ref{growth_diff2}} and \textbf{TX} hold then
		
		$$\sqrt{N}\left(U\left(\mu_N \right)- U\left(\mu \right)  \right) \overset{d}{\Longrightarrow} \mathcal{N}\left(\int_{\mathbb{R}^d}\frac{\delta U}{\delta m}\left(\mu,x\right)\sigma(dx),\int_{\mathbb{R}^d}  \left( \frac{\delta U}{\delta m}\left(\mu,x \right)\right) ^2\mu(dx)-\int_{\mathbb{R}^d \times \mathbb{R}^d} \frac{\delta U}{\delta m} (\mu,x)\frac{\delta U}{\delta m} (\mu,y)\eta(dx,dy)  \right). $$
	\end{teo}

	\begin{remark} \label{conv_nuntomu}
		Since the marginals of $\frac{1}{N}\sum_{i=1}^N\nu_i(dx)\nu_i(dy)$ and $\eta(dx,dy)$ are respectively $\bar{\nu}_N$ and $\mu$, one has $$W_\ell\left(\bar{\nu}_N,\mu\right) \leq W_\ell\left(\frac{1}{N}\sum_{i=1}^N\nu_i(dx)\nu_i(dy),\eta(dx,dy)\right)$$

		and so Assumption \ref{wess.conv2} implies that $W_\ell(\bar{\nu}_N,\mu)\underset{N\rightarrow \infty}{\longrightarrow}0$.
	\end{remark}
	\begin{remark}
		According to the theorem, the asymptotic variance is given by
		$$ \int_{\mathbb{R}^d}\left(\frac{\delta U}{\delta m} (\mu,x)\right) ^2\mu(dx)-\int_{\mathbb{R}^{2d}} \frac{\delta U}{\delta m} (\mu,x)\frac{\delta U}{\delta m} (\mu,y)\eta(dx,dy).$$
		By Jensen's inequality
		$$\frac{1}{N}\sum_{i=1}^N \int_{\mathbb{R}^{2d}} \frac{\delta U}{\delta m} (\mu,x) \frac{\delta U}{\delta m} (\mu,y)\nu_i(dx)\nu_i(dy) = \frac{1}{N}\sum_{i=1}^N \left( \int_{\mathbb{R}^d}\frac{\delta U}{\delta m} (\mu,x)\nu_i(dx)\right) ^2\ge \left( \frac{1}{N}\sum_{i=1}^N \int_{\mathbb{R}^d}  \frac{\delta U}{\delta m}(\mu,x)\nu_i(dx)\right)^2.$$
		Thanks to Hypothesis \ref{wess.conv2}, Hypothesis \textbf{RU\ref{growth2}} and Hypothesis \textbf{RU\ref{continuity}}, by taking the limit over $N\rightarrow \infty$, one deduces that
		$$
		\int_{\mathbb{R}^{2d}} \frac{\delta U}{\delta m} (\mu,x)\frac{\delta U}{\delta m} (\mu,y)\eta(dx,dy)\ge \left( \int_{\mathbb{R}^d} \frac{\delta U}{\delta m} (\mu,x)\mu(dx)\right) ^2.$$
		
		Therefore the following ``variance'' inequality holds
		$$ \int_{\mathbb{R}^d} \left( \frac{\delta U}{\delta m} (\mu,x)\right) ^2\mu(dx)-\int_{\mathbb{R}^{2d}} \frac{\delta U}{\delta m} (\mu,x)\frac{\delta U}{\delta m} (\mu,y)\eta(dx,dy)\leq \int_{\mathbb{R}^d} \left( \frac{\delta U}{\delta m} (\mu,x)\right) ^2\mu(dx)- \left( \int_{\mathbb{R}^d} \frac{\delta U}{\delta m} (\mu,x) \mu(dx)\right)^2. $$
	\end{remark}
In what follows, we provide an example where Hypotheses \ref{wess.conv2} and \ref{l_convergence} are verified.
\begin{ex}
	Let $\theta_0, \theta_1, \cdots, \theta_{m-1} \in \mathcal{P}_\ell(\mathbb{R}^d)$ for $m\in \mathbb{N}^*$ and  define  $\nu_{i}(dx):=\theta_{(i-1)\,mod\,m}(dx)$ for $i\geq 1$. We are now going to verify that Hypothesis \ref{wess.conv2} and \ref{l_convergence} are satisfied with $\eta(dx,dy)=\frac{1}{m}\sum_{i=0}^{m-1}\theta_i(dx)\theta_i(dy)$ and $\sigma=0$.\\
	If we prove that 
	\begin{equation}\label{strongproof}
		\lim_{N\to\infty}\sqrt{N}\left\|\frac{1}{N}\sum_{i=1}^{N}\theta_{(i-1)\,mod\,m}(dx)\theta_{(i-1)\,mod\,m}(dy)- \frac{1}{m}\sum_{i=0}^{m-1}\theta_i(dx)\theta_i(dy)\right\|_\ell  =0,
	\end{equation}
	then using that the convergence with respect to  $\left\|\cdot\right\|_\ell$ implies the convergence with respect to $W_\ell$, Hypothesis \ref{wess.conv2} will follow. Moreover since the marginals of $\frac{1}{N}\sum_{i=1}^{N}\theta_{(i-1)\,mod\,m}(dx)\theta_{(i-1)\,mod\,m}(dy)$ and $\frac{1}{m}\sum_{i=0}^{m-1}\theta_i(dx)\theta_i(dy)$ are respectively $\frac{1}{N}\sum_{i=1}^{N}\theta_{(i-1)\,mod\,m}$ and $\frac{1}{m}\sum_{i=0}^{m-1}\theta_i$, Hypothesis \ref{l_convergence} will follow too. Let us therefore prove (\ref{strongproof}).
	For each $N>0$, there exist $k_N\in \mathbb{N}$ and $r_N\in \mathbb{N}$ with $0\leq r_N < m$ such that $N=k_Nm+r_N$. Let $f$ be a function on $\mathbb{R}^{2d}$ such that $|f(x,y)|\leq 1+|x|^\ell+|y|^\ell$. One has
	\begin{align*}
		&\sqrt{N}\left|   \int_{\mathbb{R}^{2d}}f(x,y)\frac{1}{N}\sum_{i=1}^{N}\theta_{i-1\,mod\,m}(dx)\theta_{i-1\,mod\,m}(dy)-\frac{1}{m}\sum_{i=0}^{m-1}\int_{\mathbb{R}^{2d}}f(x,y)\theta_i(dx)\theta_i(dy)\right| \\ 
		&\phantom{==}=\sqrt{N} \left| \left( \frac{k_N}{N}-\frac{1}{m}\right) \sum_{i=0}^{m-1}\int_{\mathbb{R}^{2d}}f(x,y) \theta_{i}(dx)\theta_{i}(dy)+\frac{1}{N}\sum_{i=0}^{r_N-1}\int_{\mathbb{R}^{2d}}f(x,y) \theta_i(dx)\theta_i(dy)\right| \\
		&\phantom{==}=\left|  -\frac{r_N}{\sqrt{N}m}\sum_{i=0}^{m-1}\int_{\mathbb{R}^{2d}}f(x,y) \theta_i(dx)\theta_i(dy)+\frac{1}{\sqrt{N}}\sum_{i=0}^{r_N-1}\int_{\mathbb{R}^{2d}}f(x,y) \theta_i(dx)\theta_i(dy) \right| .
	\end{align*}
	Taking the superior limit over $N\rightarrow \infty$, the right hand-side converges to $0$ which completes the proof. 
\end{ex}
In the following example Hypothesis \ref{l_convergence} is verified with $\sigma=0$.
\begin{ex}
	Let $\mu\in \mathcal{P}_\ell(\mathbb{R}^d)$ and $\nu_{i}\in \mathcal{P}_\ell(\mathbb{R}^d)$  such that $\left\|\nu_{i} - \mu \right\|_\ell\leq \dfrac{c}{i^\alpha}$ for $i\geq 1$ with $c <\infty$ and $\alpha>\frac{1}{2}$. Then 
	\begin{align*}
		&\sqrt{N}\sup_{ f:|f(x)|\leq 1+|x|^\ell}\left|\dfrac{1}{N}\sum_{i=1}^{N}\int_{\mathbb{R}^d}f(x)\nu_{i}(dx)- \int_{\mathbb{R}^d}f(x)\mu(dx)\right|\\
		&\phantom{==} \leq \dfrac{1}{\sqrt{N}} \sum_{i=1}^{N}\left\|\nu_{i} - \mu \right\|_\ell\leq\dfrac{c}{\sqrt{N}}\sum_{i=1}^{N}\dfrac{1}{i^\alpha} = \begin{cases}
			O\left( \frac{1}{\sqrt{N}}\right)  \, &\mbox{if}\,\,\alpha>1\\
			O\left( \frac{\ln N}{\sqrt{N}}\right)  \,&\mbox{if}\,\, \alpha=1\\
			O\left( N^{\frac{1}{2}-\alpha}\right)  \, &\mbox{if}\,\,\alpha\in(0,1)\\
		\end{cases}	
	\end{align*}
	
	Therefore, since $\alpha >\frac{1}{2}$, taking the superior limit over $N\rightarrow \infty$ we have  $\sqrt{N}\|\bar{\nu}_N-\mu\|_\ell \underset{N\rightarrow \infty}{\longrightarrow}0.$ 
\end{ex}
	\subsection{Markov Chains} 
	\subsubsection{Markov Chains and the Poisson equation}
	We are now going to state the Central Limit Theorem for Markov chains but before doing it, let us recall some facts useful to it. See \cite{notesbenj} for more details. \\
	Let us consider $X_i,\, i\geq 1$ a Markov chain with initial distribution $\nu_1$ and transition kernel $P$ on $\left(\mathbb{R}^d, \mathcal{B}\left(\mathbb{R}^d\right) \right) $.\\
	We say that $P$ verifies the \textbf{Lyapunov Condition} if:
	\debutL
	\item \label{D1} $\exists V:\mathbb{R}^d \rightarrow \mathbb{R}_+$ measurable, $\exists K \in \mathbb{R}_+$, $\exists \gamma\in\left(0,1 \right)$, $\forall x\in \mathbb{R}^d$, $PV\left(x\right)\leq \gamma V\left(x\right)+K$
	\item \label{D2} $\exists R>\dfrac{2K}{1-\gamma}$,  $\exists \rho \in (0,1] $, $\forall x, y \in \mathbb{R}^d$ such that $V(x)+V(y)\leq R$, $P(x,\cdot)\wedge P(y,\cdot)(\mathbb{R}^d)\geq \rho.$
	\finL
	
	We can introduce the following normed space
	
	$$\mathcal{V}_{V}:= \left\lbrace \phi:\mathbb{R}^d \rightarrow \mathbb{R} \,\,measurable : \sup_{x\in \mathbb{R}^d}\frac{\left|\phi(x) \right|}{1+ V(x)} <\infty  \right\rbrace$$
	
	equipped with the norm
	
	$$ \|\phi \|_{V,\beta} = \sup_{x\in \mathbb{R}^d}\frac{|\phi(x)|}{1+\beta V(x)} $$
	where $\beta>0$. 
	We can also associate the distance $d_{V,\beta}$ on $\mathcal{P}_V\left(\mathbb{R}^d \right):= \left\lbrace \theta \in \mathcal{P}\left(\mathbb{R}^d\right) : \theta(V)<\infty \right\rbrace $ defined by 
	$$d_{V,\beta}\left( \theta,\sigma\right) = \sup_{\phi:\|\phi \|_{V,\beta}\leq 1} |\theta\left(\phi\right)-\sigma\left(\phi\right)|.$$
	It can be proved that under the Lyapunov condition, the transition kernel $P$ admits a unique invariant probability measure $\mu$. Moreover   
	\begin{equation}\label{int_of_V}
		\mu(V)\leq \frac{K}{1-\gamma}
	\end{equation} 
	and
	\begin{equation}\label{stima_distanza}
		\forall \beta> 0, \forall \sigma \in \mathcal{P}\left(\mathbb{R}^d\right), \forall n\in \mathbb{N}, d_{V,\beta}\left(\sigma P^n,\mu \right)\leq \left( \chi\left(\rho,\beta,\gamma,K,R \right)\right)^n d_{V,\beta}\left(\sigma,\mu \right)
	\end{equation}
	where $\chi\left(\rho,\beta,\gamma,K,R \right) = \left( 1-\rho +\beta K\right)\vee \frac{2+\beta \gamma R+2 \beta K}{2+\beta R} \in \left(0,1 \right) $ if $\beta \in \left(0,\frac{\rho}{K} \right).$
	
	The proof of this result is given in \cite{Hairer} under a stronger condition on $P$: it verifies \textbf{L\ref{D1}} and there exists a constant $q \in \left(0,1 \right)$ and a probability measure $\zeta$ so that $\inf_{x\in C}P\left(x,\cdot \right)\geq q \zeta\left(\cdot \right) $ with $C=\left\lbrace x\in \mathbb{R}^d:V\left(x \right) \leq R \right\rbrace $ for some $R>\frac{2K}{1-\gamma}$. It remains valid in our context too.\\
	We are now ready to recall the Strong Law of Large Numbers under the Lyapunov condition, a result that will be largely used in what follows.
	
	\begin{teo}\label{large_numbers}
		Let us assume that the transition kernel $P$ satisfies the Lyapunov condition and let $\mu$ be its unique invariant probability measure. Then for each function $f:\mathbb{R}^d\rightarrow \mathbb{R}$ measurable and such that $\mu(|f|)<\infty$,
		
		$$\forall \nu_1 \in \mathcal{P}(\mathbb{R}^d)\quad \mathbb{P}_{\nu_1}\left(\frac{1}{N}\sum_{k=1}^{N}f(X_k)\underset{N\rightarrow \infty}{\rightarrow} \mu(f) \right)=1.$$
	\end{teo}
	
	Since, by  (\ref{int_of_V}), $\mu(V)<\infty$ then 
	
	\begin{equation} \label{check}
		\sup_{x\in \mathbb{R}^d} \frac{|f(x)|}{1+V(x)} <\infty
	\end{equation}
	
	is a sufficient condition to ensure that $\mu(|f|)<\infty$.

	Before enunciating the Central Limit Theorem for Markov chains, we need to introduce some facts about the Poisson equation.
	For a fixed $f$ such that $\mu(|f|)<\infty$, a function $ F:\mathbb{R}^d\rightarrow \mathbb{R}$ measurable and such that $\forall x\in \mathbb{R}^d$, $P|F|(x)<\infty$ is called solution of the Poisson equation if it satisfies
	
	\begin{equation} \label{eq_poisson}
		\forall x\in \mathbb{R}^d,\quad F(x)-PF(x)=f(x)-\mu(f)
	\end{equation}
	where $\mu$ as above denotes the invariant probability measure associated to $P$.\\
	If the transition kernel $P$ satisfies the Lyapunov condition and $f\in \mathcal{V}_{V}$, the series of general term $\left(P^nf-\mu(f)\right)_{n\in \mathbb{N}}$ converges in the space $\mathcal{V}_{V}$
	equipped with the norm $\left\|  \right\|_{V,1} $ and
	\begin{equation}\label{poisson}
		F= \sum_{n\in \mathbb{N}}\left(P^nf-\mu(f)\right)
	\end{equation}
	
	is a solution of the Poisson equation (\ref{eq_poisson}). Moreover any solution can be written as $c+\sum_{n\in \mathbb{N}}\left(P^nf-\mu(f)\right)$ for a constant $c\in \mathbb{R}$.\\

	We now strengthen the condition on $P$ in order to have  the Lyapunov condition satisfied with $\sqrt{V}$ as well as with $V$. It can be proved that if $P$ satisfies \textbf{L\ref{D1}} and
	
	\begin{description} 
		\item[L2'] \label{D2'} $\exists R>\dfrac{4K}{\left( 1-\sqrt{\gamma}\right) ^2}$,  $\exists \rho \in (0,1] $, $\forall x, y \in \mathbb{R}^d$ such that $V(x)+V(y)\leq R$, $P(x,\cdot)\wedge P(y,\cdot)(\mathbb{R}^d)\geq \rho,$
	\end{description}
	then one has the Lyapunov condition for the quadruple $\left(V,\gamma,K,R\right)$ and for the quadruple $\left(\sqrt{V},\sqrt{\gamma},\sqrt{K},\sqrt{R}\right).$ In particular \textbf{L\ref{D1}} and Jensen's inequality imply that
	
	\begin{equation}\label{sqrtV}
		P\sqrt{V}(x)\leq \sqrt{PV(x)}\leq\sqrt{\gamma V(x)+K}\leq \sqrt{\gamma} \sqrt{V(x)}+\sqrt{K}.
	\end{equation}
	The following proposition holds.
	\begin{prop} \label{poisson_sqrt}
		Let us assume that the transition kernel $P$ verifies \textbf{L\ref{D1}} and \textbf{L2'}. Then if $f$ is such that $f^2\in \mathcal{V}_{V},$ $F=\sum_{n\in
			\mathbb{N}}(P^nf-\mu(f))$ converges in $\left\| \right\|_{\sqrt{V},1}$. $F$ is solution of the Poisson equation (\ref{eq_poisson}) and satisfies $F^2\in \mathcal{V}_{V}.$  
		Moreover for each $\beta \in (0,\frac{\rho}{\sqrt{K}})$ and for each $n\in \mathbb{N}$ 
		$$ \sup_{x\in \mathbb{R}^d} \frac{\left|P^nf(x) - \mu(f) \right|}{1+\sqrt{V(x)}}\leq D \chi^n \sup_{x\in \mathbb{R}^d}\frac{|f(x)|}{1+\beta \sqrt{V(x)}}  $$ 
		with $\chi = \chi(\rho, \beta,\sqrt{\gamma}, \sqrt{K},\sqrt{R})\in (0,1)$ and $D= D(\beta,\sqrt{K},\sqrt{\gamma})$ a finite constant.
		
	\end{prop}
	We are now ready to enunciate the Central Limit Theorem for Markov chains in the linear case (see as reference book \cite{CLTMC}).
	
	\begin{teo}
		Let us assume that the transition kernel $P$ verifies \textbf{L\ref{D1}} and \textbf{L2'}.  Then for each function $f: \mathbb{R}^d\rightarrow \mathbb{R}$ measurable and such that  $f^2\in \mathcal{V}_{V},$
		
		$$\sqrt{N}\left(\frac{1}{N}\sum_{i=1}^{N}f\left(X_i \right)-\int_{\mathbb{R}^d}f(x)\mu(dx)  \right) \overset{d}{\Longrightarrow } \mathcal{N}\left(0,\mu\left(F^2\right)- \mu((PF)^2)\right) $$
		with $F$ solution of the Poisson equation
		$$F(x)-PF(x)=f(x)-\mu(f) \quad x\in \mathbb{R}^d. $$ 
	\end{teo}

	\subsubsection{Statement of the theorem}
	As in the statement of Theorem \ref{independent_case}, for a given measure $\mu\in\mathcal{P}_\ell\left(\mathbb{R}^d \right)$, let us consider the Hyphoteses \textbf{RU\ref{functional_derivative2}-\ref{derivative_conv2}}, together with 
	
	\debutRU
	\item  \label{growth_diff} 
	$\exists \alpha\in (\frac{1}{2},1], \exists C < \infty, \forall \mu_1,\mu_2\in B(\mu,r)$
	$$\sup_{x\in \mathbb{R}^d}
	\frac{\left| \dfrac{\delta U}{\delta m}(\mu_2,x)-\dfrac{\delta U}{\delta m}(\mu_1,x)\right| }{1+|x|^{\frac{\ell}{2}}}
	\leq C\left( \int_{\mathbb{R}^d}(1+|y|^{\frac{\ell}{2\alpha}}) |\mu_2-\mu_1|(dy) \right)^\alpha.$$
	\finRU	
	\debutL
	\item $C_{\ell} :=  \sup_{x\in \mathbb{R}^d} \dfrac{|x|^{\ell}}{1+V(x)}< \infty.$
	\finL
	\begin{remark}
		Let us observe that by (\ref{int_of_V}) and \textbf{L3},
		we have 
		\begin{equation} \label{l_moment_finite}
			\int_{\mathbb{R}^d} |x|^\ell \mu(dx)< \infty.
		\end{equation}
	\end{remark}
	We are now ready to provide the Strong Law of Large Numbers and the Central Limit Theorem in this context.
	
	\begin{teo}[LLN for Markov chains]\label{llnmarkov}
		Let $\ell\geq 0$ and let $X_i, \,i\geq1$ a Markov chain with initial law $\nu_1$ and transition kernel $P$ that we assume to satisfy \textbf{L\ref{D1}}, \textbf{L2} and \textbf{L3}. Let $\mu$ denote its unique invariant probability measure and let us define 
		$$\mu_N :=\frac{1}{N}\sum_{i=1}^{N} \delta_{X_i}.$$ 
		Then $$W_\ell\left(\mu_N, \mu\right) \underset{N}{\longrightarrow} 0 \quad a.s.$$
		\begin{proof}
			Let us first prove that $\mu_{N} \rightharpoonup \mu$ a.s. It can be proved that for probability measures  it is possible to test the weak convergence over the continuos functions with compact support (see for instance Corollary 30.9 \cite{Bauer}). Therefore 
			$$\mathbb{P} \left( \mu_{N}(f)\underset{N}{\rightarrow} \mu(f)\,\, \forall f\in C_b(\mathbb{R}^d)\right) = \mathbb{P} \left( \mu_{N}(f)\underset{N}{\rightarrow} \mu(f)\,\, \forall f\in C_c(\mathbb{R}^d)\right) $$
			
			where $C_b(\mathbb{R}^d)$ denotes the set of bounded continuos functions and $C_c(\mathbb{R}^d)$  the set of continuos functions with compact support. Since  $C_c(\mathbb{R}^d)$ is separable with respect to the infinity norm, we can apply  Theorem \ref{large_numbers} to deduce that the right hand-side is equal to $1$ and so the almost sure weak convergence of $\mu_{N}$ to $\mu$ is proved. If $\ell >0$, to conclude the proof of the Wasserstein convergence we need to prove the convergence of the $\ell th$ moment. Since (\ref{l_moment_finite}) holds,
			we can again apply Theorem \ref{large_numbers} and deduce that
			$$ \lim_{N\rightarrow \infty}\int_{\mathbb{R}^d}|x|^\ell\mu_{N}(dx)  = \int_{\mathbb{R}^d}|x|^\ell\mu(dx) \quad a.s..$$
		\end{proof}
	\end{teo}
	\begin{teo}[CLT for Markov chains] \label{main_teo}
		Under the same notations of Theorem \ref{llnmarkov},
		let us assume that the transition kernel $P$ satisfies \textbf{L\ref{D1}}, \textbf{L2'}  and \textbf{L3}. Let us morevore assume \textbf{RU\ref{functional_derivative2}-\ref{derivative_conv2}}, \textbf{RU\ref{growth_diff}}.\\
		Then 
		$$ \sqrt{N}\left(U\left(\mu_{N} \right)-U\left(\mu \right)\right) \overset{d}{\Longrightarrow} \mathcal{N}\left(0, \mu\left(F^2\left(\mu,\cdot\right) \right) -  \mu\left( \left( PF\right)^2\left(\mu,\cdot \right)\right) \right)  $$
		
		where $F\left( \mu,\cdot\right)$ denotes the unique (up to an additive constant) solution of the Poisson equation
		$$F(\mu,x) - PF(\mu,x) = \frac{\delta U}{\delta m}\left(\mu,x \right)-\int_{\mathbb{R}^d} \frac{\delta U}{\delta m}\left(\mu,y \right)\mu(dy), \quad x\in \mathbb{R}^d.$$  
	\end{teo}
In what follows, we provide an example of functional satisfying the Hypotheses \textbf{RU\ref{functional_derivative2}-\ref{derivative_conv2}}, \textbf{RU\ref{growth_diff}}.
	\begin{ex}(U-statistics) 
	Let $\ell> 0$, $n\in\mathbb{N}\setminus \left\lbrace0,1 \right\rbrace $ and $\phi:(\mathbb{R}^{d})^n\rightarrow \mathbb{R}$ be a symmetric continuous function such that 
	\begin{equation}\label{ipotesi_phi}
	\lim_{|x_1|\rightarrow \infty}	\sup_{(x_2,\cdots,x_n)\in (\mathbb{R}^d)^{n-1}}\dfrac{|\phi(x_1,x_2,\cdots,x_n)|}{\prod_{i=1}^{n}(1+|x_i|^\frac{\ell}{2})}=0.
	\end{equation}

	We consider the function on $\mathcal{P}_\ell(\mathbb{R}^d)$ defined by 
	$$U(\mu):=\int_{(\mathbb{R}^{d})^n}\phi(x_1,\cdots,x_n)\mu(dx_1)\cdots\mu(dx_n).$$
	It is possible to prove (see Example 2.7 \cite{Jourdain}) that its linear functional derivative is defined for each $\mu\in W_\ell(\mathbb{R}^d)$ and it is given by 
	$$\frac{\delta U}{\delta m}(\mu,x_1)=n\int_{(\mathbb{R}^{d})^{n-1}} (\phi(x_1,x_2,\cdots,x_n)-\phi(0,x_2,\cdots,x_n))\mu(dx_2)\cdots\mu(dx_n), \,\, x_1\in \mathbb{R}^d.$$

	By the symmetry of $\phi$ and  (\ref{ipotesi_phi}), $\forall \epsilon>0$ $\exists M_{\epsilon}$ such that for $|x_i|>  M_{\epsilon}$
	
	\begin{equation} \label{deflimite}
		\sup_{(x_1,,\cdots,x_{i-1},x_{i+1},\cdots,x_n)\in (\mathbb{R}^d)^{n-1}}\dfrac{|\phi(x_1,x_2,\cdots,x_n)|}{\prod_{j=1}^{n}(1+|x_j|^\frac{\ell}{2})} \leq \epsilon.
	\end{equation}

Moreover by (\ref{deflimite}) and the continuity of $\phi$, $\forall$ $(x_1,\cdots,x_n)\in (\mathbb{R}^d)^n$ one has 
	\begin{align}
		\left| \phi(x_1,x_2,\cdots,x_n)\right| &=\left| \phi(x_1,x_2,\cdots,x_n)\right| 1_{\left\|x\right\|_\infty > M_1 } + \left| \phi(x_1,x_2,\cdots,x_n)\right|  1_{\left\|x\right\|_\infty \leq M_1 }\\
		& \leq \prod_{i=1}^{n}(1+|x_i|^\frac{\ell}{2}) + C \label{estimation}
	\end{align}
	for some positive constant $C<\infty$. 
	We are now going to verify that  \textbf{RU\ref{growth2}-\ref{derivative_conv2}} and \textbf{RU\ref{growth_diff}} are satisfied.\\
	Let us fix $\mu\in \mathcal{P}_\ell(\mathbb{R}^d)$ and $r>0$. We recall that $B(\mu,r)$ denotes the ball centered at $\mu$ with radius $r$ for the metric $W_\ell$.
	\begin{itemize}
		\item [\textbf{RU2.}]
	 Let $(\tilde{\mu},x_1) \in B(\mu,r)\times \mathbb{R}^d$. Using the above estimation we obtain
		\begin{align*}
			\left|\dfrac{\delta U}{\delta m}(\tilde{\mu},x_1) \right|&\leq n\int_{(\mathbb{R}^{d})^{n-1}}\left( \left| \phi(x_1,x_2,\cdots,x_n)\right|+ \left| \phi(0,x_2,\cdots,x_n)\right|\right) \tilde{\mu}(dx_2)\cdots\tilde{\mu}(dx_n)\\
			&\leq n\int_{(\mathbb{R}^{d})^{n-1}}\left( \prod_{i=1}^{n}(1+|x_i|^\frac{\ell}{2})+ \prod_{i=2}^{n}(1+|x_i|^\frac{\ell}{2})\right) \tilde{\mu}(dx_2)\cdots\tilde{\mu}(dx_n)+ 2Cn\\
			&\leq D(1+|x_1|^\frac{\ell}{2}) 
		\end{align*}
		for some positive constant $D<\infty$.
		\item[\textbf{RU3.}]  If $W_\ell(\tilde{\mu},\mu)$ goes to $0$, let $\pi\in \mathcal{P}(\mathbb{R}^{2d})$ with $\pi\left(\cdot\times \mathbb{R}^{d} \right)=\tilde{\mu}(\cdot)$, $\pi\left(\mathbb{R}^{d}\times \cdot \right)=\mu(\cdot)$.  One has
		\begin{align}
			\sup_{x_1\in \mathbb{R}^d}& \dfrac{|\frac{\delta U}{\delta m}(\tilde{\mu},x_1)-\frac{\delta U}{\delta m}(\mu,x_1)|}{1+|x_1|^\frac{\ell}{2}}\\
			&\leq n  \sup_{x_1\in \mathbb{R}^d}\dfrac{\int_{(\mathbb{R}^{d})^{2(n-1)}} \left| \phi(x_1,x_2,\cdots,x_n)- \phi(x_1,s_2,\cdots,s_n)\right| \prod_{i=2}^{n}\pi(dx_i,ds_i)}{1+|x_1|^\frac{\ell}{2}}  \\
			&+n \int_{(\mathbb{R}^{d})^{2(n-1)}} \left| \phi(0,x_2,\cdots,x_n)- \phi(0,s_2,\cdots,s_n)\right| \prod_{i=2}^{n}\pi(dx_i,ds_i)\\
			&\leq 2n  \sup_{x_1\in \mathbb{R}^d}\dfrac{\int_{(\mathbb{R}^{d})^{2(n-1)}} \left| \phi(x_1,x_2,\cdots,x_n)- \phi(x_1,s_2,\cdots,s_n)\right| \prod_{i=2}^{n}\pi(dx_i,ds_i)}{1+|x_1|^\frac{\ell}{2}}. \label{termine_conv}
		\end{align} 
	Let $\epsilon >0$. By observing that for $\left|y\right|\geq \frac{1}{\epsilon}$, $\frac{1}{\left|y\right|}\leq \epsilon$, let us define $\tilde{M}_\epsilon:=\max( \frac{1}{\epsilon}, M_{\epsilon})$. The function $\phi$ is uniformly continuous on $B=\left\lbrace x\in(\mathbb{R}^d)^{n}: \left\| x\right\|_\infty\leq \tilde{M}_\epsilon\right\rbrace$. Therefore $\exists \eta_{\epsilon,\tilde{M}_\epsilon}>0$ such that for $x, \tilde{x} \in B$ satisfying $\max_{2\leq i \leq n}|x_i - \tilde{x}_i|\leq \eta_{\epsilon,\tilde{M}_\epsilon}$,
	
	\begin{equation}\label{unifcontinuity}
		\left|\phi(x_1,x_2,\cdots,x_n)-\phi(x_1,\tilde{x}_2,\cdots,\tilde{x}_n) \right|\leq \epsilon. 
	\end{equation}
 If we denote the vectors  $(x_2,\cdots,x_n)$ and $(s_2,\cdots,s_n)$  respectively by $x_{2:n}$ and $s_{2:n}$, it is possible to rewrite the integral in (\ref{termine_conv}) in the following way ($\star$)
\begin{align*}
	\int_{}& \left| \phi(x_1,x_{2:n})- \phi(x_1,s_{2:n})\right|1_{\left\|(x_1,x_{2:n})\right\|_\infty\leq \tilde{M}_\epsilon }1_{\left\|(x_1,s_{2:n})\right\|_\infty\leq \tilde{M}_\epsilon}1_{\left\| x_{2:n} - s_{2:n}\right\|_\infty \leq \eta_{\epsilon,\tilde{M}_\epsilon}} \prod_{i=2}^{n}\pi(dx_i,ds_i)\\
	& \phantom{=}+2\int_{} \left| \phi(x_1,x_{2:n})\right|1_{\left\|(x_1,x_{2:n})\right\|_\infty\leq \tilde{M}_\epsilon }1_{\left\|(x_1,s_{2:n})\right\|_\infty\leq \tilde{M}_\epsilon}1_{\left\| x_{2:n} - s_{2:n}\right\|_\infty> \eta_{\epsilon,\tilde{M}_\epsilon}}\prod_{i=2}^{n}\pi(dx_i,ds_i)\\
	&\phantom{=}+\int_{}\left| \phi(x_1,x_{2:n})\right|1_{\left\|(x_1,x_{2:n})\right\|_\infty\leq \tilde{M}_\epsilon }1_{\left\|(x_1,s_{2:n})\right\|_\infty> \tilde{M}_\epsilon}\prod_{i=2}^{n}\pi(dx_i,ds_i)\\
	&\phantom{=}+\int_{}\left| \phi(x_1,s_{2:n})\right|1_{\left\|(x_1,x_{2:n})\right\|_\infty\leq \tilde{M}_\epsilon }1_{\left\|(x_1,s_{2:n})\right\|_\infty> \tilde{M}_\epsilon}\prod_{i=2}^{n}\pi(dx_i,ds_i)\\
	&\phantom{=}+\int_{}\left| \phi(x_1,x_{2:n})\right| 1_{\left\|(x_1,x_{2:n})\right\|_\infty> \tilde{M}_\epsilon } \prod_{i=2}^{n}\pi(dx_i,ds_i)+\int_{}\left| \phi(x_1,s_{2:n})\right| 1_{\left\|(x_1,x_{2:n})\right\|_\infty> \tilde{M}_\epsilon } \prod_{i=2}^{n}\pi(dx_i,ds_i).
\end{align*} 

Let us now study one by one the terms appearing in the above expression ($\star$).

\begin{enumerate}
	\item  By (\ref{unifcontinuity})
	$$ \int_{} \left| \phi(x_1,x_{2:n})- \phi(x_1,s_{2:n})\right|1_{\left\|(x_1,x_{2:n})\right\|_\infty\leq \tilde{M}_\epsilon }1_{\left\|(x_1,s_{2:n})\right\|_\infty\leq \tilde{M}_\epsilon}1_{\left\| x_{2:n} - s_{2:n}\right\|_\infty \leq \eta_{\epsilon,\tilde{M}_\epsilon}} \prod_{i=2}^{n}\pi(dx_i,ds_i) \leq \epsilon. $$
	\item By the continuity of $\phi$ and Markov's inequality
	\begin{align*}
2&\int_{} \left| \phi(x_1,x_{2:n})\right|1_{\left\|(x_1,x_{2:n})\right\|_\infty\leq \tilde{M}_\epsilon }1_{\left\|(x_1,s_{2:n})\right\|_\infty\leq \tilde{M}_\epsilon}1_{\left\| x_{2:n} - s_{2:n}\right\|_\infty> \eta_{\epsilon,\tilde{M}_\epsilon}}\prod_{i=2}^{n}\pi(dx_i,ds_i)\\
&\leq 2\cdot\sup_{ \left\|x\right\|_\infty\leq \tilde{M}_\epsilon }\left| \phi(x)\right|\cdot\sum_{i=2}^{n}\int_{}1_{|x_i - s_i|> \eta_{\epsilon,\tilde{M}_\epsilon}}\pi(dx_i,ds_i)\leq \frac{2 }{\eta_{\epsilon,\tilde{M}_\epsilon}^\ell}\cdot\sup_{ \left\|x\right\|_\infty\leq \tilde{M}_\epsilon}\left| \phi(x)\right|\cdot \sum_{i=2}^{n}\int_{}|x_i - s_i|^\ell\pi(dx_i,ds_i).
	\end{align*}
\item By  (\ref{estimation}), Markov's inequality, Cauchy-Schwarz inequality and the fact that $\frac{1}{\tilde{M}_\epsilon}\leq \epsilon$ 
\begin{align*}
	\int_{}&\left| \phi(x_1,x_{2:n})\right|1_{\left\|(x_1,x_{2:n})\right\|_\infty\leq \tilde{M}_\epsilon }1_{\left\|(x_1,s_{2:n})\right\|_\infty> \tilde{M}_\epsilon}\prod_{i=2}^{n}\pi(dx_i,ds_i) \leq \sum_{i=2}^{n}\int_{}(C+\prod_{j=1}^{n} (1+|x_j|^\frac{\ell}{2})) 1_{|s_i| > \tilde{M}_\epsilon}\prod_{j=2}^{n}\pi(dx_j,ds_j) \\
	&\leq\epsilon^\ell(n-1)C\int_{}|s|^\ell \mu(ds)+\epsilon^\frac{\ell}{2}(n-1)(1+|x_1|^\frac{\ell}{2})\left(\int_{}(1+|x|^\frac{\ell}{2})^2\tilde{\mu}(dx) \right)^\frac{n-1}{2}\left( \int |s|^\ell \mu(ds)\right) ^\frac{1}{2}. 
\end{align*}
\item Similarly to the previous point 
\begin{align*}
	\int_{}&\left| \phi(x_1,s_{2:n})\right|1_{\left\|(x_1,x_{2:n})\right\|_\infty\leq \tilde{M}_\epsilon }1_{\left\|(x_1,s_{2:n})\right\|_\infty> \tilde{M}_\epsilon}\prod_{i=2}^{n}\pi(dx_i,ds_i)\\
	&\leq \sum_{i=2}^{n}\int_{}(C+(1+|x_1|^\frac{\ell}{2})\prod_{j=2}^{n} (1+|s_j|^\frac{\ell}{2})) 1_{|s_i| > \tilde{M}_\epsilon}\prod_{j=2}^{n}\pi(dx_j,ds_j)\\
	&\leq \epsilon^\ell(n-1)C\int|s|^\ell \mu(ds)+\epsilon^\frac{\ell}{2}(n-1)(1+|x_1|^\frac{\ell}{2})\left(\int_{}(1+|s|^\frac{\ell}{2})^2\mu(ds) \right)^\frac{n-1}{2}\left( \int |s|^\ell \mu(ds)\right) ^\frac{1}{2}
\end{align*}
\item  Finally by (\ref{deflimite})
\begin{align*}
	\int_{}&\left| \phi(x_1,x_{2:n})\right| 1_{\left\|(x_1,x_{2:n})\right\|_\infty> \tilde{M}_\epsilon } \prod_{i=2}^{n}\pi(dx_i,ds_i)+\int_{}\left| \phi(x_1,s_{2:n})\right| 1_{\left\|(x_1,x_{2:n})\right\|_\infty> \tilde{M}_\epsilon } \prod_{i=2}^{n}\pi(dx_i,ds_i)\\
	&\leq \epsilon(1+|x_1|^\frac{\ell}{2})\left( \int_{}\left( 1+|x|^\frac{\ell}{2}\right) \tilde{\mu}(dx)\right)^{n-1} +\epsilon(1+|x_1|^\frac{\ell}{2})\left( \int_{}\left( 1+|s|^\frac{\ell}{2}\right) \mu(ds)\right)^{n-1}.
\end{align*}
\end{enumerate}
 Therefore plugging all these estimations in (\ref{termine_conv}), choosing $\pi$ as the optimal $W_\ell$ coupling between $\tilde{\mu}$ and $\mu$, we obtain
 \begin{align*}
 	\sup_{x_1\in \mathbb{R}^d}& \dfrac{|\frac{\delta U}{\delta m}(\tilde{\mu},x_1)-\frac{\delta U}{\delta m}(\mu,x_1)|}{1+|x_1|^\frac{\ell}{2}}\\
 	&\leq D \left( \sup_{ \left\|x\right\|_\infty\leq \tilde{M}_\epsilon }\left| \phi(x)\right|\cdot  \frac{W^{\ell\wedge 1}_\ell(\tilde{\mu},\mu)}{\eta_{\epsilon,\tilde{M}_\epsilon}^\ell} + (\epsilon+\epsilon^\frac{\ell}{2}+\epsilon^\ell)\left( \left(\int_{}(1+|x|^\frac{\ell}{2})^2\tilde{\mu}(dx) \right)^\frac{n-1}{2}+1 \right)\right) 
 \end{align*} 
for a positive constant $D$ depending neither on $\epsilon$ nor on $\tilde{\mu}$.
 For fixed $\epsilon$, let $W_\ell(\tilde{\mu},\mu)$ converge to $0$. Then $\int_{}| x|^\ell\tilde{\mu}(dx)$  converges to  $\int_{}| s|^\ell\mu(ds)$ and $\int_{}(1+|x|^\frac{\ell}{2})^2\tilde{\mu}(dx)$  converges to $\int_{}(1+|s|^\frac{\ell}{2})^2\mu(ds)$. It is then possible to conclude that the left-hand side goes to $0$ by letting $\epsilon$ go to $0$.
 
 \item[\textbf{RU6.}] Let $\mu_1,\mu_2\in B(\mu,r)$.  By (\ref{estimation}) we have
\begin{align*}
	\sup_{x_1\in \mathbb{R}^d}&
	\frac{\left| \dfrac{\delta U}{\delta m}(\mu_2,x_1)-\dfrac{\delta U}{\delta m}(\mu_1,x_1)\right| }{1+|x_1|^{\frac{\ell}{2}}}\\
	&\leq 2n\sup_{x_1\in \mathbb{R}^d} \frac{ \int_{(\mathbb{R}^{d})^{n-1}} \left| \phi(x_1,x_2,\cdots,x_n)\right|\left| \mu_2(dx_2)\cdots\mu_2(dx_n)-\mu_1(dx_2)\cdots\mu_1(dx_n)\right|  }{1+|x_1|^{\frac{\ell}{2}}}\\
	&\leq 2n\sup_{x_1\in \mathbb{R}^d} \frac{\sum_{k=2}^{n} \int_{(\mathbb{R}^{d})^{n-1}} \left(C+\prod_{i=1}^{n}(1+|x_i|^\frac{\ell}{2}) \right) \mu_2(dx_{k+1})\cdots\mu_2(dx_n)\left| \mu_2-\mu_1\right|(dx_k) \mu_1(dx_2)\cdots\mu_1(dx_{k-1}) }{1+|x_1|^{\frac{\ell}{2}}}\\
	&\leq \tilde{C}\int_{\mathbb{R}^{d}}(1+|x|^\frac{\ell}{2})\left| \mu_2-\mu_1\right|(dx)
\end{align*}
for a constant $\tilde{C}<\infty$ not depending on $(\mu_1,\mu_2)$.

	\end{itemize}
 
 	\end{ex} 
	\section{Proof of the Results}	
	In this section we will provide the proofs of Theorem \ref{independent_case} and Theorem \ref{main_teo}. We will start by proving the Markov chains case since it is more complex and then we will provide the proof of the independent and non-equidistributed case where the common parts, where possible, are not repeated.
	\subsection{Markov Chains}
	In the proof of the theorem we will need that the integral of $V$ respect to $\nu_1$ is finite. The next Lemma permits to suppose it in the proof.
	\begin{lemma} \label{int_V}
		If Theorem \ref{main_teo} holds for each $\nu_1 \in \mathcal{P}(\mathbb{R}^d)$ such that $\nu_1(V)<\infty$, then it holds for all $\nu_1 \in \mathcal{P}(\mathbb{R}^d)$.
	\end{lemma}
	
	\begin{proof}
		If $\nu_1(V)=\infty$, for $K> \inf_{x\in \mathbb{R}^d}V(x)$ let us consider $x_0\in \mathbb{R}^d$ such that $V\left(x_0\right)\leq K$ and define $\phi_K(x)= x1_{\left\lbrace V(x)\leq K\right\rbrace }+x_01_{\left\lbrace V(x)> K\right\rbrace }$.\\
		If we now consider the measure $\nu^K_1(A):= \nu_1(\phi_K^{-1}(A))$ for $A\in \mathcal{B}\left(\mathbb{R}^d \right)$, then 
		
		$$ \nu^K_1(V) = \int_{\mathbb{R}^d}V(x)\nu^K_1(dx) =\int_{\mathbb{R}^d}V(\phi_K(x))\nu_1(dx)=\int_{\mathbb{R}^d}\left( V(x)1_{\left\lbrace V(x)\leq K \right\rbrace}+V(x_0)1_{\left\lbrace V(x)> K \right\rbrace}\right)\nu_1(dx)\leq K. $$
		Moreover 
		
		\begin{align*}
			d_{TV}\left( \nu_1^K,\nu_1\right)&= \sup_{\psi:\left\| \psi\right\|_\infty\leq \frac{1}{2}}\left|\int_{\mathbb{R}^d}\psi(x)\nu_1^K(dx)-\int_{\mathbb{R}^d}\psi(x)\nu_1(dx)\right|\\
			&= \sup_{\psi:\left\| \psi\right\|_\infty\leq \frac{1}{2}}\left|\int_{\mathbb{R}^d}\left( \psi(x_0)-\psi(x)\right)1_{\left\lbrace V(x)>K \right\rbrace} \nu_1(dx)\right| \leq \nu_1\left( \left\lbrace x:V(x)>K \right\rbrace \right).
		\end{align*}

		Let us consider now the Markov chain with initial law $\nu_1^K$ and we denote by $\mathbb{P}_{\nu_1^K}$ the law of the process ($\mathbb{P}_{\nu_1}$ will denote the law of the original Markov chain). Since the initial law of a Markov chain determines the law of the entire process, one has that $d_{TV}\left(\mathbb{P}_{\nu_1},\mathbb{P}_{\nu_1^K}\right)=d_{TV}\left(\nu_1,\nu_1^K \right)$. \\
		Therefore if we are able to prove that for all bounded, continuous functions $f$ 
		$$\lim_{N\to\infty}\mathbb{E}_{\nu_1^K}\left(f\left(\sqrt{N}\left(U(\mu_N)-U(\mu) \right)  \right)  \right)   = \mathbb{E}\left(f(\mathcal{G}) \right)$$
		where $\mathcal{G}$ is a Gaussian random variable not depending on $K$, then
		
		\begin{align*}
			&\left| \mathbb{E}_{\nu_1}\left(f\left(\sqrt{N}\left(U(\mu_N)-U(\mu) \right)  \right)  \right) - \mathbb{E}\left(f(\mathcal{G}) \right)\right| \\
			&\phantom{\mathbb{E}_{\nu_1}}\leq \left| \mathbb{E}_{\nu_1}\left(f\left(\sqrt{N}\left(U(\mu_N)-U(\mu) \right)  \right)  \right) - \mathbb{E}_{\nu_1^K}\left(f\left(\sqrt{N}\left(U(\mu_N)-U(\mu) \right)  \right)  \right)\right|\\
			&\phantom{\mathbb{E}_{\nu_1}\leq}+ \left|  \mathbb{E}_{\nu_1^K}\left(f\left(\sqrt{N}\left(U(\mu_N)-U(\mu) \right)  \right)  \right)- \mathbb{E}\left(f(\mathcal{G}) \right)\right|\\
			&\phantom{\mathbb{E}_{\nu_1}}\leq 2\left\|f\right\|_\infty d_{TV}\left(\mathbb{P}_{\nu_1},\mathbb{P}_{\nu_1^K} \right)+ \left|  \mathbb{E}_{\nu_1^K}\left(f\left(\sqrt{N}\left(U(\mu_N)-U(\mu) \right)  \right)  \right)- \mathbb{E}\left(f(\mathcal{G}) \right)\right|\\
			&\phantom{\mathbb{E}_{\nu_1}}=2\left\|f\right\|_\infty d_{TV}\left(\nu_1,\nu_1^K \right)+ \left|  \mathbb{E}_{\nu_1^K}\left(f\left(\sqrt{N}\left(U(\mu_N)-U(\mu) \right)  \right)  \right)- \mathbb{E}\left(f(\mathcal{G}) \right)\right|.
		\end{align*}

		Since $d_{TV}\left(\nu_1,\nu_1^K \right)\leq\nu_1\left( \left\lbrace x:V(x)>K \right\rbrace \right)$, we can first take the superior limit over $N\rightarrow \infty$ so to obtain that
		$$ \limsup_{N\rightarrow \infty}\left| \mathbb{E}_{\nu_1}\left(f\left(\sqrt{N}\left(U(\mu_N)-U(\mu) \right)  \right)  \right) - \mathbb{E}\left(f(\mathcal{G}) \right)\right|\leq 2\left\|f\right\|_\infty\nu_1\left( \left\lbrace x:V(x)>K \right\rbrace \right).$$
		
		Then we can conclude that the left hand-side is equal to $0$ by letting $K$ to infinity.
	\end{proof}
	
	The proof of the theorem relies on the Poisson equation whose explicit solution and norm estimations are studied in the following Lemma.

	\begin{lemma} 
		Under the assumptions of the previous theorem, for a given $m\in B(\mu,r)$ let us consider the following Poisson equation
		
		\begin{equation} \label{poisson_trm}
			F(m,x) - PF(m,x) = \frac{\delta U}{\delta m}\left(m,x \right)-\int_{\mathbb{R}^d} \frac{\delta U}{\delta m}\left(m,y \right)\mu(dy), \quad x\in \mathbb{R}^d.
		\end{equation}
		Then (\ref{poisson_trm}) admits a solution given by $F(m,\cdot) =\sum_{n\in \mathbb{N}}(P^n\frac{\delta U}{\delta m}\left(m,\cdot \right)-\mu(\frac{\delta U}{\delta m}\left(m,\cdot \right)))$.\\
		Moreover 
		\begin{itemize}
			\item For $m\in B(\mu,r)$ \begin{equation} \label{norm_F}
				\left\|F\left(m,\cdot \right)  \right\|_{\sqrt{V},1}\leq \bar{C}_\frac{1}{2}
			\end{equation}
			
			\item For $m_1, m_2 \in B(\mu,r)$
			
			\begin{equation} \label{norm_diff_F_1}
				\left\|F\left(m_1,\cdot \right) - F\left(m_2,\cdot \right) \right\|_{\sqrt{V},1}  \leq \bar{C}_\frac{1}{2} \sup_{x\in \mathbb{R}^d} \frac{\left|\frac{\delta U}{\delta m} \left(m_1,x \right) -\frac{\delta U}{\delta m} \left(m_2,x \right)  \right| }{1+ \sqrt{V(x)}}
			\end{equation}
			
			and so for $\alpha \in (\frac{1}{2},1]$ 
			\begin{equation} \label{norm_diff_F}
				\left\|F\left(m_1,\cdot \right) - F\left(m_2,\cdot \right) \right\|_{\sqrt{V},1}  \leq\bar{C}_\frac{1}{2}\left( \int_{\mathbb{R}^d}(1+|y|^{\frac{\ell}{2\alpha}}) |m_2-m_1|(dy) \right)^\alpha 
			\end{equation}
			with $\bar{C}_\frac{1}{2}$ a finite constant not depending on $m$, $m_1$ and $m_2$.
		\end{itemize}
	\end{lemma}
	
	\begin{proof}
		By Proposition \ref{poisson_sqrt}, to ensure the existence of the solution, it is sufficient to check that $\left( \frac{\delta U}{\delta m}\left(m,\cdot \right)\right) ^2 \in \mathcal{V}_{V}.$
		By Hypothesis \textbf{RU\ref{growth2}} and  \textbf{L3}
		
		$$\sup_{x\in \mathbb{R}^d} \frac{\left(  \frac{\delta U}{\delta m}\left(m,x \right)\right) ^2 }{1+V(x)}\leq 2C^2\sup_{x\in \mathbb{R}^d} \frac{1+|x|^\ell}{1+V(x)} \leq 2C^2\left(1+C_\ell \right) < \infty. $$
		
		Let us now estimate the norm $\left\| \right\|_{\sqrt{V},1} $ of $F(m,\cdot)$. By Proposition \ref{poisson_sqrt}, one has
		\begin{align*}
			\left\|F\left(m,\cdot \right)  \right\|_{\sqrt{V},1} &= \lim_{n\rightarrow \infty} \left\|\sum_{k=0}^{n} P^k\frac{\delta U}{\delta m}\left(m,\cdot \right)-\mu(\frac{\delta U}{\delta m}\left(m,\cdot \right)) \right\|_{\sqrt{V},1}\\
			&\leq  \sum_{k=0}^{\infty}\left\| P^k\frac{\delta U}{\delta m}\left(m,\cdot \right)-\mu(\frac{\delta U}{\delta m}\left(m,\cdot \right)) \right\|_{\sqrt{V},1}.		 
		\end{align*}
		
		Let $\beta\in \left( 0,\frac{\rho}{\sqrt{K}}\right)$, then again by Proposition \ref{poisson_sqrt} and Hypothesis \textbf{RU\ref{growth2}}
		\begin{align*}
			\left\| P^k\frac{\delta U}{\delta m}\left(m,\cdot \right)-\mu(\frac{\delta U}{\delta m}\left(m,\cdot \right)) \right\|_{\sqrt{V},1}&= \sup_{x\in \mathbb{R}^d}\dfrac{\left| P^k\frac{\delta U}{\delta m}\left(m,x \right)-\mu(\frac{\delta U}{\delta m}\left(m,\cdot \right))\right| }{1+\sqrt{V(x)}}\leq D\chi^k \sup_{x\in \mathbb{R}^d}\frac{\left| \frac{\delta U}{\delta m}\left(m,x \right)\right| }{1+\beta \sqrt{V(x)}}\\
			& \leq D\chi^k C\left( 1+ \sup_{x\in \mathbb{R}^d}\frac{|x|^\frac{\ell}{2}}{1+\beta \sqrt{V(x)}} \right)\leq DC \chi^k \left(1+ \frac{\sqrt{C_\ell} }{\min(\beta,1)} \right)			
		\end{align*}

		where to obtain the last inequality we used \textbf{L3}. 
		
		Therefore we have obtained that 
		$$ \left\|F\left(m,\cdot \right)  \right\|_{\sqrt{V},1} \leq \frac{1}{1-\chi} DC \left(1+ \frac{\sqrt{C_\ell} }{\min(\beta,1)} \right). $$ 
		Finally let $m_1, m_2 \in B(\mu,r)$ and proceeding as above, by Proposition \ref{poisson_sqrt} we have
		
		$$\left\|F\left(m_1,\cdot \right) - F\left(m_2,\cdot \right) \right\|_{\sqrt{V},1} \leq  \sum_{k=0}^{\infty}\left\|P^k\left( \frac{\delta U}{\delta m}\left(m_1,\cdot \right)-\frac{\delta U}{\delta m}\left(m_2,\cdot \right) \right)- \mu\left(\frac{\delta U}{\delta m}\left(m_1,\cdot \right)-\frac{\delta U}{\delta m}\left(m_2,\cdot \right) \right)  \right\|_{\sqrt{V},1} $$
		
		Let $\beta \in (0,\frac{\rho}{\sqrt{K}})$, then  by Proposition \ref{poisson_sqrt} and Hypothesis \textbf{RU\ref{growth_diff}}\\
		
		\begin{align*}
			&\left\|P^k\left( \frac{\delta U}{\delta m}\left(m_1,\cdot \right)-\frac{\delta U}{\delta m}\left(m_2,\cdot \right) \right)- \mu\left(\frac{\delta U}{\delta m}\left(m_1,\cdot \right)-\frac{\delta U}{\delta m}\left(m_2,\cdot \right) \right)  \right\|_{\sqrt{V},1}\\
			&\phantom{P^k\quad\quad}= \sup_{x\in \mathbb{R}^d} \dfrac{\left| P^k\left( \frac{\delta U}{\delta m}\left(m_1,x \right)-\frac{\delta U}{\delta m}\left(m_2,x \right) \right)- \mu\left(\frac{\delta U}{\delta m}\left(m_1,\cdot \right)-\frac{\delta U}{\delta m}\left(m_2,\cdot \right) \right)\right| }{1+\sqrt{V(x)}}\\
			&\phantom{P^k\quad\quad}\leq D\chi^k \sup_{x\in \mathbb{R}^d}\frac{\left| \frac{\delta U}{\delta m}\left(m_1,x \right)-\frac{\delta U}{\delta m}\left(m_2,x \right)\right|} {1+\beta \sqrt{V(x)}}\\
			&\phantom{P^k\quad\quad}\leq DC\chi^k\sup_{x\in \mathbb{R}^d} \frac{1+|x|^{\frac{\ell}{2}}}{1+\beta \sqrt{V(x)}}\left( \int_{\mathbb{R}^d}(1+|y|^{\frac{\ell}{2\alpha}}) |m_2-m_1|(dy) \right)^\alpha\\ &\phantom{P^k\quad\quad}=DC\chi^k\left(1+\frac{\sqrt{C_\ell}}{\min\left(\beta,1\right) } \right) \left( \int_{\mathbb{R}^d}(1+|y|^{\frac{\ell}{2\alpha}}) |m_2-m_1|(dy) \right)^\alpha
		\end{align*}
		
		where we used  \textbf{L3} to obtain the last equality. Therefore we have
		
		$$ \left\|F\left(m_1,\cdot \right) - F\left(m_2,\cdot \right) \right\|_{\sqrt{V},1}\leq DC\frac{1}{1-\chi}\left(1+\frac{\sqrt{C_\ell}}{\min\left(\beta,1\right) } \right)\left( \int_{\mathbb{R}^d}(1+|y|^{\frac{\ell}{2\alpha}}) |m_2-m_1|(dy) \right)^\alpha. $$
	\end{proof}
	
	\begin{proof}[Proof of Theorem \ref{main_teo}]
		\textbf{FIRST STEP}\quad
		
		Let us preliminary recall that by Theorem \ref{llnmarkov},  $W_\ell\left(\mu_N, \mu\right) \underset{N}{\longrightarrow} 0$  a.s.
		To study the limit distribution of $\sqrt{N}\left(U\left(\mu_N \right)- U\left(\mu \right)\right)$, let us define for $i=1,\cdots,N$ and $s\in \left[ 0,1\right] $
		
		$$\mu_{N}^{i,s} = \frac{1}{N} \sum_{j=1}^{i-1}\delta_{X_j}+\frac{s}{N}\delta_{X_i}+\left(1+\frac{1-i-s}{N} \right)\mu.$$ 
		We have $\mu_{N}^{N,1}= \mu_{N} $ and
		$\mu_{N}^{1,0}=\mu$. Moreover, since for $i=1,\cdots,N-1$ 
		$$\mu_{N}^{i,1} = \mu_{N}^{i+1,0}, $$
		one has
		
		$$\sqrt{N}\left(U\left(\mu_N \right)- U\left(\mu\right)  \right) = \sqrt{N} \sum_{i=1}^{N}\left(U(\mu_{N}^{i,1}) -U(\mu_{N}^{i,0})  \right) .$$

		Let us now show that almost surely $\mu_{N}^{i,s}$  converges uniformly to $\mu$ with respect to the distance $W_\ell$ that is
		
		\begin{equation} \label{unif_conv}
			\max_{1\leq i\leq N} \sup_{s\in \left[0,1 \right]} W_\ell(\mu_{N}^{i,s},\mu) \underset{N\rightarrow\infty}{\rightarrow} 0 \quad a.s..
		\end{equation}

		For  $i=1,\cdots,N$ and $s\in \left[0,1 \right]$
		
		$$ \mu_{N}^{i,s} = s\mu_{N}^{i,1}+(1-s)\mu_{N}^{i-1,1}$$
		under the convention $\mu_{N}^{0,1}=\mu_{N}^{1,0}$. Let $\pi \in \Pi(\mu_{N}^{i,1},\mu)$ and $\tilde{\pi}\in \Pi(\mu_{N}^{i-1,1},\mu)$ where $\Pi(\mu_1,\mu_2)= \left\lbrace \mu \in \mathcal{P}(\mathbb{R}^{2d}): \mu(\cdot\times \mathbb{R}^{d})=\mu_1(\cdot),\mu(\mathbb{R}^{d}\times\cdot)= \mu_2(\cdot)\right\rbrace $ for $\mu_1, \mu_2 \in\mathcal{P}_\ell(\mathbb{R}^{d})$ and define
		
		$$ \bar{\pi}\left(dx,dy \right)= s\pi\left(dx,dy \right) + (1-s)\tilde{\pi}\left(dx,dy \right).$$
		Then
		
		\begin{align*}
			&\bar{\pi}\left(dx,\mathbb{R}^d \right)= s\pi\left(dx,\mathbb{R}^d \right) + (1-s)\tilde{\pi}\left(dx,\mathbb{R}^d \right) = s\mu_{N}^{i,1}+(1-s)\mu_{N}^{i-1,1} = \mu_{N}^{i,s}\\
			&\bar{\pi}\left(\mathbb{R}^d,dy \right)= s\pi\left(\mathbb{R}^d,dy \right) + (1-s)\tilde{\pi}\left(\mathbb{R}^d,dy \right) = s\mu + (1-s)\mu = \mu.
		\end{align*}
		
		By supposing that $\ell>0$, we have:   
		\begin{equation}\label{magg0}
					 W_\ell^{\ell\vee 1}\left(\mu_{N}^{i,s},\mu \right)\leq \int_{\mathbb{R}^d \times \mathbb{R}^{d}} |x-y|^\ell \bar{\pi}\left(dx,dy \right) = s\int_{\mathbb{R}^d \times \mathbb{R}^{d}} |x-y|^\ell \pi\left(dx,dy \right) + (1-s)\int_{\mathbb{R}^d \times \mathbb{R}^{d}} |x-y|^\ell \tilde{\pi}\left(dx,dy \right).
		\end{equation}

		Taking the infimum over $\pi$ and $\tilde{\pi}$, we conclude that
		
		$$  W_\ell^{\ell\vee 1}\left(\mu_{N}^{i,s},\mu \right)\leq sW_\ell^{\ell\vee 1}\left(\mu_{N}^{i,1},\mu \right) + (1-s)W_\ell^{\ell\vee 1}\left(\mu_{N}^{i-1,1},\mu \right)\leq W_\ell^{\ell\vee 1}\left(\mu_{N}^{i,1},\mu \right) \vee  W_\ell^{\ell\vee 1}\left(\mu_{N}^{i-1,1},\mu \right) $$  
		
		and so
		
		\begin{equation}\label{unifconv}
			\max_{1\leq i\leq N} \sup_{s\in \left[0,1 \right]} W_\ell\left( \mu_{N}^{i,s},\mu\right)  = \max_{0\leq i\leq N}W_\ell\left( \mu_{N}^{i,1},\mu \right).  
		\end{equation}
		Now since for $i=1,\cdots, N$
		$$ \mu_{N}^{i,1} = \frac{1}{N}\sum_{j=1}^{i}\delta_{X_j}+\left(1-\frac{i}{N} \right)\mu = \frac{i}{N}\mu_{i} + \left(1-\frac{i}{N} \right)\mu, $$
		
		let $\gamma\in \Pi(\mu_{i},\mu)$ and define
		$ \tilde{\gamma}(dx,dy)=\frac{i}{N}\gamma(dx,dy) +\left(1-\frac{i}{N} \right)\mu\left(dx \right) \delta_x(dy).$
		Then
		
		\begin{align*}
			&\tilde{\gamma}(dx,\mathbb{R}^d) = \frac{i}{N}\gamma(dx,\mathbb{R}^d)+\left(1-\frac{i}{N} \right)\mu\left(dx \right) \delta_x(\mathbb{R}^d) = \frac{i}{N}\mu_{i}(dx)+\left(1-\frac{i}{N} \right)\mu\left(dx \right) =  \mu_{N}^{i,1}(dx)\\
			&\tilde{\gamma}(\mathbb{R}^d,dy) = \frac{i}{N}\gamma(\mathbb{R}^d,dy)+\left(1-\frac{i}{N} \right)\mu\left(dy \right)\delta_y(\mathbb{R}^d) = \frac{i}{N}\mu(dy)+ \left(1-\frac{i}{N} \right)\mu(dy) = \mu(dy).
		\end{align*}
		
		Therefore
		
		\begin{equation}\label{magg}
			W_\ell^{\ell\vee1}\left( \mu_{N}^{i,1},\mu\right)\leq \int_{\mathbb{R}^d\times\mathbb{R}^d}|x-y|^\ell\tilde{\gamma}(dx,dy) = \frac{i}{N}\int_{\mathbb{R}^d\times\mathbb{R}^d}|x-y|^\ell \gamma(dx,dy).
		\end{equation} 
		
		Taking the infimum over $\gamma$, one has
		$$W_\ell\left( \mu_{N}^{i,1},\mu\right)\leq \left( \frac{i}{N}\right) ^\frac{1}{\ell\vee1}W_\ell\left( \mu_i,\mu\right). $$ 
		
		Finally let $\alpha\in\left(0,1 \right) $ so that
		
		\begin{align*}
			\max_{0\leq i\leq N}W_\ell\left( \mu_{N}^{i,1},\mu \right)&\leq \max_{1\leq i\leq N} \left( \frac{i}{N}\right) ^\frac{1}{\ell\vee1} W_\ell\left( \mu_i,\mu \right) = \max_{1\leq i\leq \llcorner\alpha N\lrcorner} \left( \frac{i}{N}\right) ^\frac{1}{\ell\vee1} W_\ell\left( \mu_i,\mu \right) +\max_{\llcorner\alpha N\lrcorner< i\leq N}\left( \frac{i}{N}\right) ^\frac{1}{\ell\vee1} W_\ell\left( \mu_i,\mu \right)\\
			&\leq \alpha^\frac{1}{\ell\vee1}\sup_{i} W_\ell\left( \mu_i,\mu \right) + \max_{\llcorner\alpha N\lrcorner< i\leq N} W_\ell\left( \mu_i,\mu \right).
		\end{align*}
		
		Since we have proved that $\lim_{N\to\infty}W_\ell(\mu_{N},\mu)=0$, for fixed $\alpha$ the last term goes to $0$ as $N$ goes to infinity while the first one is arbitrarily small for $\alpha$ small and so (\ref{unif_conv}) is proved.\\\\
		By replacing $\left|x-y\right|^\ell$ with $1\wedge \left|x-y\right|$ in (\ref{magg0}) and (\ref{magg}), we can do exactly the same for $\ell=0$. \\\\
		
		\textbf{SECOND STEP}\quad Let us define

		$$I_N = \min\left\lbrace 1\leq i \leq N : \exists s\in\left[ 0,1\right] :W_\ell\left(\mu^{i,s}_{N},\mu \right)\geq r  \right\rbrace  $$
		
		and let us introduce the filtration $\left(\mathcal{F}_i=\sigma\left(X_1,\cdots,X_i\right)\right)_{i\geq1} $ for which $I_N$ is a stopping time. According to the first step, under the assumption $\min\emptyset=N+1$, $I_N$ is $a.s.$ equal to $N+1$  for each $ N\geq N^*$ for a random variable $N^*$ taking integers values. This stopping time allows to introduce in the proof the linear functional derivative associated to $U$ since by Hypothesis \textbf{RU\ref{functional_derivative2}} it is well defined in the ball of radius $r$ and center $\mu$.  \\
		
		For $N\geq N^*$, by Lemma \ref{linearderivative}, we have
		$$U\left(\mu_N \right)- U\left(\mu\right) =  \sum_{i=1}^{N}\left(U(\mu_{N}^{i,1}) -U(\mu_{N}^{i,0})  \right) =  \sum_{i=1}^{N} \int_{0}^{1}ds\int_{\mathbb{R}^d} \frac{\delta U}{\delta m}(\mu_{N}^{i,s},x)  \frac{(\delta_{X_i}-\mu)(dx)}{N}. $$
		
		Setting
		
		$$Q_N :=\frac{1}{N}\sum_{i=1}^{N}\left( \frac{\delta U}{\delta m}(\mu_{N}^{i\wedge I_N,0},X_i) -\int_{\mathbb{R}^d} \frac{\delta U}{\delta m}(\mu_{N}^{i\wedge I_N,0},x)\mu(dx) \right),$$
		
		we deduce that for $N\geq N^*$, $U\left(\mu_N \right)- U\left(\mu\right) - Q_N$ coincides with
		
		$$R_N = \frac{1_{N\geq N^*}}{N}\sum_{i=1}^{N}  \int_{0}^{1}ds\int_{\mathbb{R}^d} \left( \frac{\delta U}{\delta m}(\mu_{N}^{i,s},x)- \frac{\delta U}{\delta m}(\mu_{N}^{i,0},x)\right)  (\delta_{X_i}-\mu)(dx).$$
		
		Let us therefore consider the following decomposition:
		$$\sqrt{N}\left(U\left(\mu_N \right)- U\left(\mu\right)  \right) = \sqrt{N}\left(U\left(\mu_N \right)- U\left(\mu\right) - Q_N \right) + \sqrt{N}Q_N. $$
		We will see that the first term will go to $0$ in probability (\textbf{third step}) while the second one will converge in distribution to a normal random variable (\textbf{fourth step}). \\
		
		\textbf{THIRD STEP}\quad Let us first prove the convergence in probability of $\sqrt{N}\left(U\left(\mu_N \right)- U\left(\mu\right) - Q_N \right)$ to $0$.
		By definition we need to prove that $\forall \epsilon>0$
		
		$$\lim_{N\to\infty} \mathbb{P}\left(\sqrt{N}|U\left(\mu_N \right)- U\left(\mu\right) - Q_N| \geq \epsilon\right) = 0. $$
		
		One has 
		
		$$\left\lbrace  \sqrt{N}|U\left(\mu_N \right)- U\left(\mu\right) - Q_N| \geq \epsilon\right\rbrace$$
		$$= \left(  \left\lbrace \sqrt{N}|U\left(\mu_N \right)- U\left(\mu\right) - Q_N|\geq \epsilon \right\rbrace\bigcap \left\lbrace N\geq N^* \right\rbrace \right)   \overset{\cdot}{\bigcup} \left( \left\lbrace \sqrt{N}|U\left(\mu_N \right)- U\left(\mu\right) - Q_N|\geq \epsilon \right\rbrace\bigcap \left\lbrace N<N^* \right\rbrace \right) $$
		$$ \subseteq \left(  \left\lbrace \sqrt{N}|U\left(\mu_N \right)- U\left(\mu\right) - Q_N|\geq \epsilon \right\rbrace\bigcap \left\lbrace N\geq N^* \right\rbrace \right)  \overset{\cdot}{\bigcup} \left\lbrace N<N^* \right\rbrace. $$
		Therefore
		
		$$\mathbb{P}\left(\sqrt{N}|U\left(\mu_N \right)- U\left(\mu\right) - Q_N| \geq \epsilon\right)\leq \mathbb{P}\left(N<N^* \right)+\mathbb{P}\left(\sqrt{N}1_{N\geq N^*}|U\left(\mu_N \right)- U\left(\mu\right) - Q_N|\geq \epsilon\right). $$
		
		Since $\lim_{N\to\infty} \mathbb{P}\left(N<N^* \right) =0, $ it is sufficient to prove the almost sure convergence of $\sqrt{N}R_N$ to $0$. One has \\
		\begin{align*}
			|R_N|&\leq \frac{1_{N\geq N^*}}{N}\sum_{i=1}^{N}\int_{0}^{1}ds\int_{\mathbb{R}^d} \bigg\rvert\frac{\delta U}{\delta m}(\mu_{N}^{i,s},x)- \frac{\delta U}{\delta m}(\mu_{N}^{i,0},x)\bigg\rvert|\delta_{X_i}-\mu|(dx)\\
			&\leq \frac{1_{N\geq N^*}}{N}\sum_{i=1}^{N}\int_{0}^{1}ds\int_{\mathbb{R}^d} \bigg\rvert\frac{\delta U}{\delta m}(\mu_{N}^{i,s},x)- \frac{\delta U}{\delta m}(\mu_{N}^{i,0},x)\bigg\rvert(\delta_{X_i}+\mu)(dx).
		\end{align*}

		By Assumption \textbf{RU\ref{growth_diff}}, $\exists C < \infty$, $\exists \alpha \in \left(\frac{1}{2},1 \right] $ such that for $N\geq N^*$
		$$ \bigg\rvert\frac{\delta U}{\delta m}(\mu_{N}^{i,s},x)- \frac{\delta U}{\delta m}(\mu_{N}^{i,0},x)\bigg\rvert \leq C(1+|x|^{\frac{\ell}{2}})\left(\int_{\mathbb{R}^d}\left( 1+|y|^{\frac{\ell}{2\alpha}}\right) |\mu_{N}^{i,s}-\mu_{N}^{i,0}|(dy) \right)^\alpha $$ 
		
		with
		$$|\mu_{N}^{i,s} -\mu_{N}^{i,0}|(dy) \leq \frac{s}{N}(\delta_{X_i}+\mu)(dy).$$
		
		Substituting the above quantity and using the subadditivity of $x\mapsto x^\alpha$ one obtains
		
		\begin{align*}
			\bigg\rvert\frac{\delta U}{\delta m}(\mu_{N}^{i,s},x)- \frac{\delta U}{\delta m}(\mu_{N}^{i,0},x)\bigg\rvert  &\leq C\left((1+|x|^{\frac{\ell}{2}})\left(2+\int_{\mathbb{R}^d}|y|^{\frac{\ell}{2\alpha}}(\delta_{X_i}+\mu)(dy) \right)^\alpha \left( \frac{s}{N}\right) ^\alpha \right)\\	
			&\leq \frac{C}{N^\alpha} (1+|x|^{\frac{\ell}{2}})\left(2^\alpha+|X_i|^\frac{ \ell}{2}+\left( \int_{\mathbb{R}^d}|y|^\frac{ \ell}{2\alpha}\mu(dy)\right)^\alpha  \right)\\
			& \leq \frac{C^*}{N^\alpha}\left(1+|x|^{\frac{\ell}{2}} \right) \left(1+|X_i|^\frac{ \ell}{2} \right)
		\end{align*}

		where we used (\ref{l_moment_finite}) to obtain the last inequality with $C^*$ a finite constant.
		Therefore
		
		\begin{align*}
			|R_N| &\leq \frac{C^*}{N^{\alpha+1}} \sum_{i=1}^{N}\int_{\mathbb{R}^d}\left(1+|x|^{\frac{\ell}{2}} \right) \left(1+|X_i|^\frac{ \ell}{2} \right)(\delta_{X_i}+\mu)(dx)\\
			& = \frac{C^*}{N^{\alpha+1}}  \sum_{i=1}^{N}\left(1+|X_i|^\frac{ \ell}{2} \right) \left(2+ |X_i|^\frac{ \ell}{2} + \int_{\mathbb{R}^d}|x|^{\frac{\ell}{2}} \mu(dx) \right)\\
			& = \frac{C^*}{N^{\alpha+1}} \sum_{i=1}^{N} \left(2 +\int_{\mathbb{R}^d}|x|^{\frac{\ell}{2}} \mu(dx) + \left( 3+\int_{\mathbb{R}^d}|x|^{\frac{\ell}{2}} \mu(dx)\right) |X_i|^\frac{ \ell}{2} + |X_i|^\ell\right)\\
			&\leq \frac{C_1}{N^\alpha}+ \frac{C_2}{N^\alpha} \int_{\mathbb{R}^d}|y|^\ell\mu_{N}(dy)
		\end{align*}

		for some positive constants $C_1$ and $C_2$. Therefore one has\\

		$$\sqrt{N}|R_N| \leq \frac{C_1}{N^{\alpha-\frac{1}{2}}}+\frac{C_2}{N^{\alpha-\frac{1}{2}}}\int_{\mathbb{R}^d}|y|^\ell\mu_{N}(dy)$$ \\
		
		and since $\alpha>\frac{1}{2}$ and $\lim_{N\to\infty}\int_{\mathbb{R}^d}|y|^\ell\mu_{N}(dy) = \int_{\mathbb{R}^d}|y|^\ell\mu(dy)$, we can conclude that the left-hand side goes to $0$ as $N$ goes to infinity and so the third step is concluded. \\
		
		\textbf{FOURTH STEP}\quad As anticipated, we are now going to prove the convergence in distribution of $\sqrt{N}Q_N$ to a Gaussian random variable. \\
		Remember that 
		$$ \sqrt{N}Q_N = \frac{1}{\sqrt{N}}\sum_{i=1}^{N}\left( \frac{\delta U}{\delta m}(\mu_{N}^{i\wedge I_N,0},X_i) -\int_{\mathbb{R}^d} \frac{\delta U}{\delta m}(\mu_{N}^{i\wedge I_N,0},x)\mu(dx) \right)$$

		and let us consider for a given $m\in B(\mu,r)$ the following Poisson equation
		
		\begin{equation} \label{poisson_fd}
			F(m,x) - PF(m,x) = \frac{\delta U}{\delta m}\left(m,x \right)-\int_{\mathbb{R}^d} \frac{\delta U}{\delta m}\left(m,y \right)\mu(dy), \quad x\in \mathbb{R}^d.
		\end{equation}

		By Proposition \ref{poisson_sqrt} we know that it admits a solution.	
		Thanks to that, we are able to rewrite $\sqrt{N}Q_N$ as  
		
		\begin{align*}
			\sqrt{N}Q_N &= \frac{1}{\sqrt{N}} \sum_{i=1}^{N}\left( F\left(\mu_{N}^{i\wedge I_N,0},X_i\right) - PF\left(\mu_{N}^{i\wedge I_N,0},X_i\right)\right)\\
			& = \dfrac{1}{\sqrt{N}}\left( \sum_{i=1}^{N} F\left(\mu_{N}^{i\wedge I_N,0},X_i\right) 
			- \sum_{i=2}^{N} F\left(\mu_{N}^{(i-1)\wedge I_N,0},X_{i}\right) \right) \\
			&\phantom{=}+\frac{1}{\sqrt{N}}\left( \sum_{i=2}^{N} F\left(\mu_{N}^{(i-1)\wedge I_N,0},X_{i}\right) 
			- \sum_{i=2}^{N+1} PF\left(\mu_{N}^{(i-1)\wedge I_N,0},X_{i-1}\right)\right) \\
			&=  K_{0,N} + K_{1,N} + K_{2,N} + K_{3,N}
		\end{align*}

		with 
		
		\begin{align*}
			&K_{0,N} = \dfrac{F\left(\mu,X_1\right)}{\sqrt{N}} ,\\
			& K_{1,N} = \dfrac{1}{\sqrt{N}} \sum_{i=2}^{N}\left(F\left(\mu_{N}^{i\wedge I_N,0},X_i\right) - F\left(\mu_{N}^{(i-1)\wedge I_N,0},X_{i}\right)\right),\\
			&K_{2,N} =- \dfrac{PF\left(\mu_{N}^{N\wedge I_N,0},X_{N}\right)}{\sqrt{N}},\\
			&K_{3,N} =\dfrac{1}{\sqrt{N}}\sum_{i=2}^{N}\left( F\left(\mu_{N}^{(i-1)\wedge I_N,0},X_{i}\right) -PF\left(\mu_{N}^{(i-1)\wedge I_N,0},X_{i-1}\right)\right).	
		\end{align*}

		The idea now is to study the convergence of $ K_{i,N}$ for $i=0,\cdots, 3.$ We will see that $K_{0,N}, K_{1,N}$ and $K_{2,N}$ go to $0$ in probability as $N$ goes to infinity while $K_{3,N}$ is the term providing the convergence in distribution to a Gaussian random variable. By Slutsky's theorem, we can therefore obtain the limit of $\sqrt{N}Q_N$ (same as the limit of $K_{3,N}$) and conclude the proof. \\
		
		\underline{\textbf{Convergence of $K_{0,N}+K_{1,N}+K_{2,N}$ to $0$ in probability}}\\
		
		The almost sure convergence  (and so in probability) of $K_{0,N}$ to $0$ is immediate.\\
		For what concerns the convergence in probability of $K_{1,N}$ to $0$, following the same idea used in the third step, it is sufficient to prove the almost sure convergence of $1_{N\geq N^*}|K_{1,N}|$ to $0$. Therefore
		
		\begin{align}
			1_{N\geq N^*}\left|K_{1,N}\right|&= \left| \dfrac{1_{N\geq N^*}}{\sqrt{N}} \sum_{i=2}^{N}\left(F\left(\mu_{N}^{i\wedge I_N,0},X_i\right) - F\left(\mu_{N}^{(i-1)\wedge I_N,0},X_{i}\right)\right) \right|\\
			& \leq \dfrac{1_{N\geq N^*}}{\sqrt{N}} \sum_{i=2}^{N}\left|F\left(\mu_{N}^{i,0},X_i\right) - F\left(\mu_{N}^{i-1,0},X_{i}\right) \right|\\\label{k1N}
			&\leq \dfrac{1_{N\geq N^*}}{\sqrt{N}} \sum_{i=2}^{N} \left\|F\left(\mu_{N}^{i,0},\cdot\right) - F\left(\mu_{N}^{i-1,0},\cdot\right) \right\|_{\sqrt{V},1} \left( 1 + \sqrt{V(X_i)}\right).
		\end{align}

		By (\ref{norm_diff_F}), observing that $\mu_{N}^{i,0}-\mu_{N}^{i-1,0}= \frac{1}{N}\left( \delta_{X_{i-1}}-\mu\right), $
		one obtains
		
		\begin{align*}
			1_{N\geq N^*}\left\|F\left(\mu_{N}^{i,0},\cdot\right) - F\left(\mu_{N}^{i-1,0},\cdot\right) \right\|_{\sqrt{V},1}&\leq1_{N\geq N^*}\bar{C}_\frac{1}{2}\left( \int_{\mathbb{R}^d}\left(1+ |y|^{\frac{\ell}{2\alpha}}\right)  |\mu_{N}^{i-1,0}-\mu_{N}^{i,0}|(dy) \right)^\alpha\\
			&\leq 1_{N\geq N^*}\bar{C}_\frac{1}{2} \frac{1}{N^\alpha}  \left(2+ \int_{\mathbb{R}^d}|y|^{\frac{\ell}{2\alpha}} \delta_{X_{i-1}}(dy)+\int_{\mathbb{R}^d}|y|^{\frac{\ell}{2\alpha}}\mu(dy) \right)^\alpha\\
			&\leq 1_{N\geq N^*}\bar{C}^*_\frac{1}{2} \frac{1}{N^\alpha}\left( |X_{i-1}|^{\frac{\ell}{2}} +1 \right).
		\end{align*}

		where we used (\ref{l_moment_finite}) and the subadditivity of $x\mapsto x^\alpha$ to obtain the last inequality with $\bar{C}^*_\frac{1}{2}$ a finite constant. Therefore plugging this inequality in (\ref{k1N}), we obtain the following estimation for $1_{N\geq N^*}\left|K_{1,N}\right|$
		
		\begin{align*}
			1_{N\geq N^*}\left|K_{1,N}\right| &\leq \dfrac{1_{N\geq N^*}\bar{C}^*_\frac{1}{2}}{N^{\alpha+\frac{1}{2}}} \sum_{i=2}^{N}  \left( |X_{i-1}|^{\frac{\ell}{2}} + 1 \right)\left( 1 + \sqrt{V(X_i)}\right)\\
			&=1_{N\geq N^*}\bar{C}^*_\frac{1}{2} \left( \dfrac{1}{N^{\alpha-\frac{1}{2}}}\times \frac{1}{N}\sum_{i=2}^{N}  |X_{i-1}|^{\frac{\ell}{2}} +\dfrac{1}{N^{\alpha+\frac{1}{2}}} \sum_{i=2}^{N}|X_{i-1}|^{\frac{\ell}{2}} \sqrt{V(X_i)}\right) \\
			&\phantom{=}+ 1_{N\geq N^*}\bar{C}^*_\frac{1}{2}\left(  \frac{1}{N^{\alpha-\frac{1}{2}}}\times\frac{1}{N}\sum_{i=2}^{N}  \sqrt{V(X_i)} + \frac{N-1}{N^{\alpha+\frac{1}{2}}}\right)\\
			&\leq 1_{N\geq N^*}\bar{C}^*_\frac{1}{2} \left( \dfrac{1}{N^{\alpha-\frac{1}{2}}}\times\frac{1}{N}\sum_{i=2}^{N} |X_{i-1}|^{\ell}+\dfrac{1}{N^{\alpha-\frac{1}{2}}}\times\frac{1}{N} \sum_{i=2}^{N} V(X_i) + 2\frac{N-1}{N^{\alpha+\frac{1}{2}}}\right).
		\end{align*}

		Since $\alpha> \frac{1}{2}$, $\frac{1}{N}\sum_{i=2}^{N} |X_{i-1}|^{\ell}$ converges a.s. to $\int_{\mathbb{R}^d}|x|^\ell \mu(dx)$ and $\frac{1}{N} \sum_{i=2}^{N} V(X_i)$ converges a.s. to $\mu(V)$, we can conclude that $K_{1,N}$ goes to $0$ in probability. \\	
		Finally let us prove the convergence of $K_{2,N}$ to $0$ in probability.	Still in this case it is sufficient to prove that $1_{N\geq N^*}|K_{2,N}|$ converges in probability to $0$.
		By (\ref{norm_F}) we have
		
		\begin{align*}
			1_{N\geq N^*}|K_{2,N}| &= \frac{1_{N\geq N^*}}{\sqrt{N}}\left|PF\left(\mu_{N}^{N,0},X_{N}\right)  \right|\leq\frac{1_{N\geq N^*}}{\sqrt{N}} \left\|F\left(\mu_{N}^{N,0},\cdot \right) \right\|_{\sqrt{V},1}\int_{\mathbb{R}^d}(1+\sqrt{V(x)})P(X_N,dx)\\
			&=\frac{1_{N\geq N^*}}{\sqrt{N}} \left\|F\left(\mu_{N}^{N,0},\cdot \right) \right\|_{\sqrt{V},1}\left(1+P\sqrt{V}(X_N) \right) \leq \frac{1_{N\geq N^*}}{\sqrt{N}}\bar{C}_\frac{1}{2}\left(1+\sqrt{\gamma}\sqrt{V(X_N)}+\sqrt{K} \right)
		\end{align*}

		where we applied (\ref{sqrtV}) to obtain the last inequality. We can therefore conclude that the right-hand side goes to $0$ as $N$ goes to infinity by observing that by Theorem \ref{large_numbers}, $\bar{V}_N:=\frac{1}{N}\sum_{i=1}^{N}V\left(X_i \right)$ converges almost surely to $ \mu\left(V \right)$ and consequently 
		$$\sqrt{\frac{V\left(X_N\right) }{N}} =\sqrt{\bar{V}_N-\frac{N-1}{N}\bar{V}_{N-1}} \longrightarrow 0 \,\, a.s.$$

		\underline{\textbf{Convergence in distribution of $K_{3,N}$}}\\ 
		
		To study the convergence in distribution of $K_{3,N}$, we apply the Central Limit Theorem for martingales (see Corollary 3.1 \cite{HallHeyde}). First we need to prove a square integrable martingale property and then that the Bracket condition and the Lindeberg condition hold.\\
		Let us recall the expression of $K_{3,N}$
		$$ K_{3,N} =\dfrac{1}{\sqrt{N}}\sum_{i=2}^{N}\left( F\left(\mu_{N}^{(i-1)\wedge I_N,0},X_{i}\right) -PF\left(\mu_{N}^{(i-1)\wedge I_N,0},X_{i-1}\right)\right).$$

		For $i=2,\cdots,N$ let 
		
		$$Y_{N,i}= \dfrac{1}{\sqrt{N}}\left( F\left(\mu_{N}^{(i-1)\wedge I_N,0},X_{i}\right) -PF\left(\mu_{N}^{(i-1)\wedge I_N,0},X_{i-1}\right)\right).$$\\
		
		We will start by \textbf{checking that $\mathbb{E}\left(Y_{N,i}|\mathcal{F}_{i-1}\right)=0$.}
		Since	$\mu^{i\wedge I_N,0}_{N}=\sum_{j=1}^{i-1}1_{\left\lbrace I_N=j\right\rbrace }\mu^{j,0}_{N}+1_{\left\lbrace I_N>i-1\right\rbrace}\mu^{i,0}_{N}$ is $\mathcal{F}_{i-1}$- measurable, one has
		\begin{align*}
			\mathbb{E}\left(Y_{N,i}|\mathcal{F}_{i-1}\right) &= \dfrac{1}{\sqrt{N}}\mathbb{E}\left(F\left(\mu_{N}^{(i-1)\wedge I_N,0},X_{i}\right) -PF\left(\mu_{N}^{(i-1)\wedge I_N,0},X_{i-1}\right) |\mathcal{F}_{i-1}\right)\\
			&= \dfrac{1}{\sqrt{N}}\left(\mathbb{E}\left(F\left(\mu_{N}^{(i-1)\wedge I_N,0},X_{i}\right) |\mathcal{F}_{i-1}\right) - \mathbb{E}\left(F\left(\mu_{N}^{(i-1)\wedge I_N,0},X_{i}\right)| \mathcal{F}_{i-1}\right) 
			\right) = 0.
		\end{align*}

		For what concerns the \textbf{square integrability}, it is sufficient to check that 
		$$\mathbb{E}\left(F^2\left(\mu_{N}^{(i-1)\wedge I_N,0},X_{i}\right) \right) < \infty.$$
		
		By (\ref{norm_F}) 
		$$ \mathbb{E}\left(F^2\left(\mu_{N}^{(i-1)\wedge I_N,0},X_{i}\right) \right)\leq \mathbb{E}\left(  \left\|F\left(\mu_{N}^{(i-1)\wedge I_N,0},\cdot\right) \right\|^2_{\sqrt{V},1}\left(1+\sqrt{V\left( X_i\right)}  \right)^2\right) \leq 2\bar{C}^2_\frac{1}{2} \left(  
		1+\mathbb{E}\left( V\left( X_i\right) \right)\right)  $$
		
		where  $\mathbb{E}\left( V\left( X_i\right) \right)\leq \gamma^{i-1}\nu_1(V)+K\sum_{\ell=0}^{i-2}\gamma^\ell $ and so by Lemma \ref{int_V}, the right-hand side is finite.\\

		Let us now study the	\textbf{convergence of $\sum_{i=2}^{N}\mathbb{E}\left(Y^2_{N,i}|\mathcal{F}_{i-1} \right)$}:
		
		\begin{align*}
			&\sum_{i=2}^{N}\mathbb{E}\left(Y^2_{N,i}|\mathcal{F}_{i-1} \right)\\
			&= \frac{1}{N} \sum_{i=2}^{N} \mathbb{E}\left(
			F^2\left(\mu_{N}^{(i-1)\wedge I_N,0},X_{i}\right) +\left( PF\right)^2\left(\mu_{N}^{(i-1)\wedge I_N,0},X_{i-1}\right)
			-2 F\left(\mu_{N}^{(i-1)\wedge I_N,0},X_{i}\right)PF\left(\mu_{N}^{(i-1)\wedge I_N,0},X_{i-1}\right)|\mathcal{F}_{i-1} \right)\\
			&= \frac{1}{N} \sum_{i=2}^{N}
			PF^2\left(\mu_{N}^{(i-1)\wedge I_N,0},X_{i-1}\right)-\frac{1}{N} \sum_{i=2}^{N}\left( PF\right)^2\left(\mu_{N}^{(i-1)\wedge I_N,0},X_{i-1}\right)
		\end{align*}

		Before studying the behavior of the above quantity, let us observe that the convergence of\\ $\sup_{x\in \mathbb{R}^d}\dfrac{|\frac{\delta U}{\delta m}\left(\tilde{\mu},x \right)-\frac{\delta U}{\delta m}\left( \mu,x \right)|}{1+|x|^\frac{\ell}{2}}$ to $0$ when $W_\ell\left(\tilde{\mu}, \mu \right)$ goes to $0$ together with the a.s. convergence of $\max_{1\leq i\leq N} W_\ell(\mu_{N}^{i,0},\mu)$ to $0$ imply the existence of a sequence of random variables $\left( \epsilon_N\right)_{N\geq 0} $ converging a.s. to $0$ as $N\rightarrow \infty$ such that
		
		\begin{equation} \label{control}
			\forall 1\leq i\leq N\qquad \bigg\rvert\frac{\delta U}{\delta m}\left(\mu^{i\wedge I_N,0}_N,x \right)-\frac{\delta U}{\delta m}\left( \mu,x \right)\bigg\rvert\leq\left(1+|x|^\frac{\ell}{2}\right)\epsilon_N.
		\end{equation}
		
		With (\ref{norm_diff_F_1}) and \textbf{L3}, we deduce that
		
		\begin{equation} \label{control_diff}
			\left\|F\left(\mu^{i\wedge I_N,0}_N,\cdot \right) - F\left(\mu,\cdot \right) \right\|_{\sqrt{V},1}  \leq \bar{C}_\frac{1}{2} \epsilon_N \sup_{x\in \mathbb{R}^d} \frac{1+|x|^\frac{\ell}{2}   }{1+\sqrt{V(x)}} \leq\bar{C}_\frac{1}{2} \left(1+\sqrt{C_\ell} \right) \epsilon_N.
		\end{equation}\\

		\underline{First Component}\quad 
		We can rewrite the first component in the following way 
		
		$$\frac{1}{N} \sum_{i=2}^{N}
		PF^2\left(\mu_{N}^{(i-1)\wedge I_N,0},X_{i-1}\right) = \frac{1}{N} \sum_{i=2}^{N}\left( 
		PF^2\left(\mu_{N}^{(i-1)\wedge I_N,0},X_{i-1}\right) - PF^2\left(\mu,X_{i-1}\right)\right) +\frac{1}{N} \sum_{i=2}^{N} PF^2\left(\mu,X_{i-1}\right).  $$
		
		Since by \textbf{L\ref{D1}} \begin{equation}\label{PFsquared}
			\left\|PF^2\left(\mu,\cdot\right)\right\|_{V,1} \leq \left\|F^2\left(\mu,\cdot\right)\right\|_{V,1} \sup_{x\in \mathbb{R}^d} \frac{1+PV(x)}{1+V(x)} < \infty, 
		\end{equation}
		by Theorem \ref{large_numbers} we obtain that 
		$$ \lim_{N\rightarrow \infty}\frac{1}{N} \sum_{i=2}^{N} PF^2\left(\mu,X_{i-1}\right) = \mu(PF^2\left(\mu,\cdot\right))=\mu(F^2\left(\mu,\cdot\right)) \quad a.s. $$
		where for the last equality we used the invariance of $\mu$ with respect to $P$. On the other hand
		
		\begin{align*}
			&\left| \frac{1}{N} \sum_{i=2}^{N}\left( 
			PF^2\left(\mu_{N}^{(i-1)\wedge I_N,0},X_{i-1}\right) - PF^2\left(\mu,X_{i-1}\right)\right)\right|\\
			&\phantom{\frac{1}{N}}\leq \frac{1}{N} \sum_{i=2}^{N} \int_{\mathbb{R}^d}\left| F^2\left(\mu_{N}^{(i-1)\wedge I_N,0},x\right) - F^2\left(\mu,x\right)\right| P\left(X_{i-1},dx \right)\\
			&\phantom{\frac{1}{N}}\leq \frac{1}{N} \sum_{i=2}^{N}\int_{\mathbb{R}^d}\left(  F\left(\mu_{N}^{(i-1)\wedge I_N,0},x\right) - F\left(\mu,x\right)\right)^2  P\left(X_{i-1},dx \right)\\
			& \phantom{\frac{1}{N}}\quad+ \frac{2}{N} \sum_{i=2}^{N}\int_{\mathbb{R}^d} \left|F\left(\mu,x\right) \right|  \left|F\left(\mu_{N}^{(i-1)\wedge I_N,0},x\right) - F\left(\mu,x\right) \right|  P\left(X_{i-1},dx \right)\\
			&\phantom{\frac{1}{N}}\leq \frac{1}{N} \sum_{i=2}^{N} \left\|F\left(\mu_{N}^{(i-1)\wedge I_N,0},\cdot\right) - F\left(\mu,\cdot\right) \right\|^2_{\sqrt{V},1}\left(1+PV\left(X_{i-1}\right)+2P\sqrt{V}\left(X_{i-1}\right) \right)\\
			&\phantom{\frac{1}{N}}\quad+  \frac{2}{N} \sum_{i=2}^{N} \left\|F\left(\mu, \cdot \right)\right\|_{\sqrt{V},1} \left\|F\left(\mu_{N}^{(i-1)\wedge I_N,0},\cdot \right) - F\left(\mu, \cdot \right)\right\|_{\sqrt{V},1} \left(1+PV\left(X_{i-1}\right)+2P\sqrt{V}\left(X_{i-1}\right)  \right)\\
			&\phantom{\frac{1}{N}}\leq \epsilon_N\left(\bar{C}^2_\frac{1}{2} \left(1+\sqrt{C_\ell} \right)^2 \epsilon_N+ 2\bar{C}^2_\frac{1}{2} \left(1+\sqrt{C_\ell} \right)   \right)\frac{1}{N}\sum_{i=2}^{N} \left(1+\gamma V\left(X_{i-1}\right) + K+2\sqrt{\gamma}\sqrt{V\left(X_{i-1}\right)} +2\sqrt{K} \right)
		\end{align*}

		where we used (\ref{control_diff}), (\ref{norm_F}), the Lyapunov condition and (\ref{sqrtV}) . Therefore the right-hand side goes to $0$ since $\frac{1}{N}\sum_{i=2}^{N}V\left(X_{i-1}  \right) $ converges to $\mu(V)$, $\frac{1}{N}\sum_{i=2}^{N}\sqrt{V\left(X_{i-1}  \right)} $ converges to $\mu(\sqrt{V})$  and $\epsilon_N$ goes to $0$ a.s..\\
		
		\underline{Second Component}\quad 
		As before, we can rewrite the second component in the following way 
		\begin{align*}
			& \frac{1}{N} \sum_{i=2}^{N}\left( PF\right)^2\left(\mu_{N}^{(i-1)\wedge I_N,0},X_{i-1}\right) \\
			&\phantom{\frac{1}{N}}= \frac{1}{N} \sum_{i=2}^{N}\left( \left( PF\right)^2\left(\mu_{N}^{(i-1)\wedge I_N,0},X_{i-1}\right)-\left( PF\right)^2\left(\mu,X_{i-1}\right)\right) +\frac{1}{N} \sum_{i=2}^{N}\left( PF\right)^2\left(\mu,X_{i-1}\right).\end{align*}
		
		By (\ref{PFsquared}) and Jensen's inequality $ \left\|\left( PF\right)^2\left(\mu,\cdot\right) \right\|_{V,1} \leq \infty $, therefore by Theorem \ref{large_numbers}
		$$\frac{1}{N} \sum_{i=2}^{N}\left( PF\right)^2\left(\mu,X_{i-1}\right)  = \mu(\left( PF\right)^2\left(\mu,\cdot\right)) \quad a.s.. $$

		On the other hand
		
		\begin{align*}
			&\left| \frac{1}{N} \sum_{i=2}^{N}\left( \left( PF\right)^2\left(\mu_{N}^{(i-1)\wedge I_N,0},X_{i-1}\right)-\left( PF\right)^2\left(\mu,X_{i-1}\right)\right)\right|\\
			&\phantom{\frac{1}{N}}\leq\frac{1}{N} \sum_{i=2}^{N} \left(  PF\left(\mu_{N}^{(i-1)\wedge I_N,0},X_{i-1}\right)- PF\left(\mu,X_{i-1}\right)\right)^2\\
			&\phantom{=\frac{1}{N}}+ \frac{2}{N}\sum_{i=2}^{N}\left| PF\left(\mu,X_{i-1}\right)\right|\left| PF\left(\mu_{N}^{(i-1)\wedge I_N,0},X_{i-1}\right)- PF\left(\mu,X_{i-1}\right)\right|\\
			&\phantom{\frac{1}{N}}\leq \frac{1}{N} \sum_{i=2}^{N}\left(\left\|F\left(\mu_{N}^{(i-1)\wedge I_N,0},\cdot \right)- F\left(\mu,\cdot \right) \right\|_{\sqrt{V},1} \left( 1+P\sqrt{V}(X_{i-1})\right)  \right)^2\\
			&\phantom{\frac{1}{N}}\quad+\frac{2}{N}\sum_{i=2}^{N} \left\| F\left(\mu,\cdot \right) \right\|_{\sqrt{V},1}\left\|F\left(\mu_{N}^{(i-1)\wedge I_N,0},\cdot \right)- F\left(\mu,\cdot \right) \right\|_{\sqrt{V},1}\left( 1+P\sqrt{V}(X_{i-1})\right)^2\\
			&\phantom{\frac{1}{N}}\leq \epsilon_N\left( \bar{C}^2_\frac{1}{2} \left(1+\sqrt{C_\ell} \right)^2\epsilon_N + 2\bar{C}^2_\frac{1}{2} \left(1+\sqrt{C_\ell} \right)\right)  \frac{1}{N}\sum_{i=2}^{N}\left(1+\sqrt{\gamma}\sqrt{V(X_{i-1})} +\sqrt{K}\right)^2.
		\end{align*}

		The right-hand side goes to $0$ as $N$ goes to infinity since $\epsilon_N$ goes to $0$ a.s. and, by Theorem \ref{large_numbers}, $\frac{1}{N}\sum_{i=2}^{N}V\left(X_{i-1}  \right) $ converges to $\mu(V)$ and $\frac{1}{N}\sum_{i=2}^{N}\sqrt{V\left(X_{i-1}  \right)} $ converges to $\mu(\sqrt{V})$.\\
		In conclusion we have proved that almost surely
		
		\begin{equation}\label{asint.variance}
			\frac{1}{N} \sum_{i=2}^{N}
			\left( 	PF^2\left(\mu_{N}^{(i-1)\wedge I_N,0},X_{i-1}\right)-\left( PF\right)^2\left(\mu_{N}^{(i-1)\wedge I_N,0},X_{i-1}\right)\right) \underset{N\rightarrow \infty}{\longrightarrow} \mu\left(F^2\left(\mu,\cdot\right) \right) -  \mu\left( \left( PF\right)^2\left(\mu,\cdot \right) \right).
		\end{equation} 
		To finally conclude that $K_{3,N} = \sum_{i=2}^{N} Y_{N,i}
		\overset{d}{\Longrightarrow} \mathcal{N}\left(0, \mu\left(F^2\left(\mu,\cdot\right) \right) -  \mu\left( \left( PF\right)^2\left(\mu,\cdot \right)\right) \right)$, it remains just to verify that the \textbf{Lindeberg condition} holds.\\	We need to check that for $\epsilon>0$, $\sum_{i=2}^{N}\mathbb{E}\left(Y^2_{N,i} 1_{\left\lbrace Y^2_{N,i}> \epsilon \right\rbrace }| \mathcal{F}_{i-1}\right) $ goes to $0$ in probability as $N$ goes to infinity.
		By Jensen's inequality, the Lyapunov condition, (\ref{sqrtV}) and (\ref{norm_F})  we obtain
		
		\begin{align*}
			NY^2_{N,i} &= \left( F\left(\mu_{N}^{(i-1)\wedge I_N,0},X_{i}\right) -PF\left(\mu_{N}^{(i-1)\wedge I_N,0},X_{i-1}\right)\right)^2\\
			&\leq 2F^2\left(\mu_{N}^{(i-1)\wedge I_N,0},X_{i}\right)+2 PF^2\left(\mu_{N}^{(i-1)\wedge I_N,0},X_{i-1}\right)\\
			&\leq 2 \left\|F\left(\mu_{N}^{(i-1)\wedge I_N,0},\cdot\right) \right\|^2_{\sqrt{V},1}\left(1+\sqrt{V(X_i)} \right)^2+2\left\|F\left(\mu_{N}^{(i-1)\wedge I_N,0},\cdot\right) \right\|^2_{\sqrt{V},1} \left(1+PV(X_{i-1})+2P\sqrt{V}(X_{i-1}) \right)\\
			&\leq  2\bar{C}^2_\frac{1}{2}\left( 2+V(X_i)+2\sqrt{V\left(X_{i} \right) } +\gamma V\left(X_{i-1} \right)+K+2\sqrt{\gamma}\sqrt{V(X_{i-1})}+2\sqrt{K} \right)\\
			&\leq  2L\bar{C}^2_\frac{1}{2}\left(1+ V(X_i)+ V\left(X_{i-1} \right)\right) 
		\end{align*}
		
		where $L$ is a finite constant. As, for $a,b,c,g \in \mathbb{R}_+$

		\begin{align*}
			\left(a+b+c \right)1_{\left\lbrace a+b+c\geq g \right\rbrace}&=\left(a+b+c \right) \left(1_{\left\lbrace a>b,a>c,a+b+c\geq g \right\rbrace}+1_{\left\lbrace b\geq a,b>c,a+b+c\geq g \right\rbrace}+1_{\left\lbrace c\geq a,c\geq b,a+b+c\geq g \right\rbrace} \right)\\
			&\leq 3a1_{\left\lbrace a>b,a>c,a+b+c\geq g \right\rbrace}+3b1_{\left\lbrace b\geq a,b>c,a+b+c\geq g \right\rbrace}+3c1_{\left\lbrace c\geq a,c\geq b,a+b+c\geq g \right\rbrace}\\
			&\leq 3a1_{\left\lbrace a\geq\frac{g}{3} \right\rbrace} + 3b1_{\left\lbrace b\geq\frac{g}{3} \right\rbrace}+3c1_{\left\lbrace c\geq\frac{g}{3} \right\rbrace},
		\end{align*}

		it is enough to check that for each $\epsilon >0$ 
		
		$$\frac{N-1}{N} 1_{\left\lbrace \frac{1}{N}> \epsilon\right\rbrace }  + \sum_{i=2}^{N} \mathbb{E}\left(\frac{ V(X_i)}{N} 1_{\left\lbrace \frac{V(X_i)}{N}> \epsilon\right\rbrace } |\mathcal{F}_{i-1} \right) +   \sum_{i=2}^{N} \frac{V(X_{i-1})}{N} 1_{\left\lbrace \frac{V(X_{i-1})}{N}> \epsilon\right\rbrace }$$
		
		goes to $0$ in probability as $N$ goes to infinity.\\
		It is immediate to prove that the \textbf{first component} goes to $0$. For the \textbf{second component} let us observe that
		
		\begin{align*}
			\mathbb{E}\left(  \sum_{i=2}^{N} \mathbb{E}\left(\frac{V(X_i)}{N} 1_{\left\lbrace \frac{V(X_i)}{N}> \epsilon\right\rbrace } |\mathcal{F}_{i-1} \right)\right)  &= \sum_{i=2}^{N} \mathbb{E}\left(\frac{V(X_i)}{N} 1_{\left\lbrace \frac{V(X_i)}{N}> \epsilon\right\rbrace } \right)\\
			& = \sum_{i=2}^{N} \int_{\mathbb{R}^d} \frac{V(x)}{N} 1_{\left\lbrace \frac{V(x)}{N}> \epsilon\right\rbrace } \nu_1P^{i-1}\left( dx\right)\\
			&= \sum_{i=2}^{N} \int_{\mathbb{R}^d} \frac{V(x)}{N} 1_{\left\lbrace \frac{V(x)}{N}> \epsilon\right\rbrace } \left( \nu_1P^{i-1}\left( dx\right)-\mu\left( dx\right)  \right)  \\
			&\phantom{=}+ \frac{N-1}{N}\int_{\mathbb{R}^d} V(x) 1_{\left\lbrace \frac{V(x)}{N}> \epsilon\right\rbrace }\mu(dx)\\
			&\leq \sum_{i=2}^{N} \left\|\frac{V(\cdot)}{N} 1_{\left\lbrace \frac{V(\cdot)}{N}> \epsilon\right\rbrace }\right\|_{V,\beta} d_{V,\beta}\left( \nu_1P^{i-1},\mu \right)+ \frac{N-1}{N}\int_{\mathbb{R}^d} V(x) 1_{\left\lbrace \frac{V(x)}{N}> \epsilon\right\rbrace }\mu(dx).
		\end{align*}

		Choosing $\beta \in \left(0,\frac{\rho}{\sqrt{K}} \right)$, we can apply (\ref{stima_distanza})  and deduce that
		\begin{align*}
			&\sum_{i=2}^{N} \left\|\frac{V(\cdot)}{N} 1_{\left\lbrace \frac{V(\cdot)}{N}> \epsilon\right\rbrace }\right\|_{V,\beta} d_{V,\beta}\left( \nu_1P^{i-1},\mu \right)+ \frac{N-1}{N}\int_{\mathbb{R}^d} V(x) 1_{\left\lbrace \frac{V(x)}{N}> \epsilon\right\rbrace }\mu(dx)\\
			&\phantom{==}\leq \left\|\frac{V(\cdot)}{N} 1_{\left\lbrace \frac{V(\cdot)}{N}> \epsilon\right\rbrace }\right\|_{V,\beta} \sum_{i=0}^{\infty} \chi^i  d_{V,\beta}\left(\nu_1,\mu \right)+ \frac{N-1}{N}\int_{\mathbb{R}^d} V(x) 1_{\left\lbrace \frac{V(x)}{N}> \epsilon\right\rbrace }\mu(dx)\\
			& \phantom{==}\leq \frac{1}{N}\sup_{x\in \mathbb{R}^d} \frac{V(x)}{1+\beta V(x)} \frac{1}{1-\chi}d_{V,\beta}\left(\nu_1,\mu \right)+ \frac{N-1}{N}\int_{\mathbb{R}^d} V(x) 1_{\left\lbrace \frac{V(x)}{N}> \epsilon\right\rbrace }\mu(dx).
		\end{align*}
		By Lemma \ref{int_V}, it is enough to suppose that $\nu_1(V)<\infty$ so that $d_{V,\beta}\left(\nu_1,\mu \right)<\infty$ and the first term goes to $0$.
		The second term goes to $0$ by Lebesgue's theorem since $\mu(V)<\infty$.
		Finally let us study the \textbf{third component}. 
		By Theorem \ref{large_numbers} we know that $\overline{V}_N:= \frac{1}{N}\sum_{i=1}^{N}V\left( X_{i}\right)$ converges almost surely to $\mu\left(V \right)$. Therefore 
		
		$$\frac{V\left( X_{N}\right)}{N} = \overline{V}_N - \frac{N-1}{N}\overline{V}_{N-1}\underset{N\rightarrow \infty}{\longrightarrow} 0 \quad a.s.$$
		
		that implies $\forall \epsilon >0$
		
		$$V\left( X_{N}\right) 1_{\left\lbrace   \frac{V\left( X_{N}\right)}{N} > \epsilon \right\rbrace} \underset{N\rightarrow \infty}{\longrightarrow} 0 \quad a.s.$$ 
		
		and taking the Cesaro mean we deduce that
		
		$$\lim_{N\to\infty} \frac{1}{N} \sum_{i=1}^{N}  V\left( X_{i}\right) 1_{\left\lbrace \frac{V\left( X_{i}\right)}{N} > \epsilon\right\rbrace } \leq \lim_{N\to\infty} \frac{1}{N} \sum_{i=1}^{N}V\left( X_{i}\right)1_{\left\lbrace   \frac{V\left( X_{i}\right)}{i} > \epsilon \right\rbrace} = 0.  $$

	\end{proof}

	\subsection{Independent Non-Equidistributed Random Variables}
	Before providing the proof of Theorem \ref{independent_case}, let us observe that Hypothesis \textbf{TX }implies the Lindeberg condition.
	\begin{lemma} Hypothesis \textbf{TX} implies that
		\begin{equation}\label{lindenberg}
			\forall \epsilon>0,\underset{N\to\infty}{\lim}\dfrac{1}{N}\mathlarger{\sum}_{i=1}^{N}\mathbb{E}\left(|X_i|^\ell 1_{\left\lbrace |X_i|^\ell>N\epsilon \right\rbrace} \right)= 0.
		\end{equation}	
	\end{lemma}
	\begin{proof}
		An application of Kronecker's lemma shows that \textbf{TX} implies 
		\begin{equation} \label{sci}
			\lim_{N\to\infty} \dfrac{1}{N}\mathlarger{\sum}_{i=1}^{N}\mathbb{E}\left( \left( |X_i|^\ell-i^\beta\right)  1_{\left\lbrace |X_i|^\ell>i^\beta \right\rbrace} \right)=0.
		\end{equation}
		
		Let $\epsilon>0$, then exists $\bar{N}_{\epsilon,\beta}$ such that for $N\geq \bar{N}_{\epsilon,\beta},$  $N^\beta \leq \dfrac{\epsilon N}{2}$. Therefore if $N\geq \bar{N}_{\epsilon,\beta},$ the following chain of inequalities holds for each $i=1,\cdots,N$
		\begin{align*}
			\left( |X_i|^\ell-i^\beta\right)  1_{\left\lbrace |X_i|^\ell>i^\beta \right\rbrace} &\geq\left( |X_i|^\ell-\dfrac{\epsilon N}{2}\right)  1_{\left\lbrace |X_i|^\ell>\epsilon N \right\rbrace}\\
			&\geq \dfrac{|X_i|^\ell}{2}1_{\left\lbrace |X_i|^\ell>\epsilon N \right\rbrace}.
		\end{align*}
		
		Taking the expectation   and using (\ref{sci}) we can obtain (\ref{lindenberg}).
		
	\end{proof}
	\begin{proof}[Proof of Theorem \ref{independent_case}]

		To study the limit distribution of $\sqrt{N}\left(U\left(\mu_N \right)- U\left(\mu \right)  \right)$, let us consider the following decomposition
		\begin{equation} \label{decomposition}
			\sqrt{N}\left(U\left(\mu_N \right)- U\left(\mu \right)  \right) = \sqrt{N}\left(U\left(\mu_N \right)- U\left(\bar{\nu}_N\right)  \right)+ \sqrt{N}\left(U\left(\bar{\nu}_N \right)- U\left(\mu\right)  \right).
		\end{equation}
		We will prove that $\sqrt{N}\left(U\left(\mu_N \right)- U\left(\bar{\nu}_N\right)  \right) \overset{d}{\Rightarrow} \mathcal{N}\left(0,\int_{\mathbb{R}^d}  \frac{\delta U}{\delta m}\left(\mu,x \right)^2\mu(dx)-\int_{\mathbb{R}^d \times \mathbb{R}^d} \frac{\delta U}{\delta m} (\mu,x)\frac{\delta U}{\delta m} (\mu,y)\eta(dx,dy)  \right)$ and $\sqrt{N}\left(U\left(\bar{\nu}_N \right)- U\left(\mu\right)  \right)\underset{N\rightarrow \infty}{\longrightarrow} \int_{\mathbb{R}^d}\frac{\delta U}{\delta m}(\mu,x)\sigma(dx).$\\
		Let us start by studying the second term in the above decomposition. 		
		
		\subsection*{Limit of $\sqrt{N}\left(U\left(\bar{\nu}_N \right)- U\left(\mu\right)\right)$}
		
		Let us first observe that by Remark \ref{conv_nuntomu}, Assumption \ref{wess.conv2} implies $
		W_\ell\left(\bar{\nu}_N, \mu\right) \underset{N\rightarrow \infty}{\longrightarrow} 0$. Therefore one has
		$$\lim_{N\to\infty} \sup_{s\in \left[0,1 \right]} W_\ell\left(\mu+s(\bar{\nu}_N - \mu),\mu \right)=0.$$
		
		For $N$ bigger than a fixed $N_r$, by Lemma \ref{linearderivative} we can rewrite
		$$\sqrt{N}\left(U\left(\bar{\nu}_N\right)- U\left(\mu\right)  \right) = \sqrt{N}\int_{0}^{1}ds\int_{\mathbb{R}^d}\frac{\delta U}{\delta m}(\mu+s(\bar{\nu}_N - \mu),x)(\bar{\nu}_N-\mu)(dx).$$
		Let us now prove that the above quantity tends to $\int_{\mathbb{R}^d}\frac{\delta U}{\delta m}(\mu,x)\sigma(dx)$ where $\sigma$ is the measure such that $\|\sqrt{N}\left(\bar{\nu}_N-\mu \right) - \sigma\|_\ell \underset{N\rightarrow \infty}{\longrightarrow} 0$ (see Hypothesis \ref{l_convergence}).
		By the triangle inequality, for $N\geq N_r$, one has
		
		\begin{align*}
			&\left|\sqrt{N}\int_{0}^{1}ds\int_{\mathbb{R}^d}\frac{\delta U}{\delta m}(\mu+s(\bar{\nu}_N - \mu),x)(\bar{\nu}_N-\mu)(dx) - \int_{\mathbb{R}^d}\frac{\delta U}{\delta m}(\mu,x)\sigma(dx) \right|\\
			& \phantom{==}\leq \left| \sqrt{N}\int_{0}^{1}ds\int_{\mathbb{R}^d}\frac{\delta U}{\delta m}(\mu+s(\bar{\nu}_N - \mu),x)(\bar{\nu}_N-\mu)(dx) - \sqrt{N}\int_{\mathbb{R}^d}\frac{\delta U}{\delta m}(\mu,x)(\bar{\nu}_N-\mu)(dx)\right|\\
			&\phantom{===}+ \left|\sqrt{N}\int_{\mathbb{R}^d}\frac{\delta U}{\delta m}(\mu,x)(\bar{\nu}_N-\mu)(dx) - \int_{\mathbb{R}^d}\frac{\delta U}{\delta m}(\mu,x)\sigma(dx) \right| .
		\end{align*}

		Recalling that $\left\|\tau \right\|_\ell= \sup_{f:|f(x)|\leq 1+|x|^\ell} \int_{\mathbb{R}^d}f(x)\tau(dx)$ for $\tau \in \mathcal{M}_\ell(\mathbb{R}^d)$, the second component of the right-hand side goes to $0$ thanks to Assumption \ref{l_convergence} and Assumption \textbf{RU\ref{growth2}}.\\
		For what concerns the first component of the right-hand side, we have by Hypothesis \textbf{RU\ref{growth_diff2}}
		
		\begin{align*}
			&\left| \sqrt{N}\int_{0}^{1}ds\int_{\mathbb{R}^d}\frac{\delta U}{\delta m}(\mu+s(\bar{\nu}_N - \mu),x)(\bar{\nu}_N-\mu)(dx) - \sqrt{N}\int_{\mathbb{R}^d}\frac{\delta U}{\delta m}(\mu,x)(\bar{\nu}_N-\mu)(dx)\right|\\
			&\leq \int_{0}^{1}ds\int_{\mathbb{R}^d} \left| \frac{\delta U}{\delta m}(\mu+s(\bar{\nu}_N - \mu),x) - \frac{\delta U}{\delta m}(\mu,x) \right|\left|\sqrt{N}\left(\bar{\nu}_N-\mu \right)  \right|(dx) \\
			&\leq \int_{0}^{1}ds\int_{\mathbb{R}^d} C\left( (1+|x|^\ell)\left\| s(\bar{\nu}_N-\mu)\right\|_0^\alpha +(1+|x|^{\ell (1-\alpha)})\left( \int_{\mathbb{R}^d}|y|^\ell |s(\bar{\nu}_N-\mu)|(dy) \right)^\alpha \right)\left| \sqrt{N}\left(\bar{\nu}_N-\mu \right)  \right|(dx)\\
			&\leq C \left( \left|\bar{\nu}_N-\mu \right| (\mathbb{R}^d)\right)^\alpha\cdot\int_{\mathbb{R}^d}(1+|x|^\ell)\left| \sqrt{N}\left(\bar{\nu}_N-\mu \right)  \right|(dx)\\
			&\phantom{=}+ C \left( \int_{\mathbb{R}^d}|y|^\ell |\bar{\nu}_N-\mu|(dy) \right)^\alpha \int_{\mathbb{R}^d}(1+|x|^{\ell (1-\alpha)})\left| \sqrt{N}\left(\bar{\nu}_N-\mu \right)  \right|(dx).
		\end{align*}

		Since $\|\sqrt{N}\left(\bar{\nu}_N-\mu \right) - \sigma\|_\ell \underset{N\rightarrow \infty}{\longrightarrow} 0$ implies
		$\||\sqrt{N}\left(\bar{\nu}_N-\mu \right)| - |\sigma|\|_\ell \underset{N\rightarrow \infty}{\longrightarrow} 0$, we can deduce that:
		
		\begin{enumerate}[label=(\roman*)]
			\item $\int_{\mathbb{R}^d}(1+|x|^\ell)| \sqrt{N}\left(\bar{\nu}_N-\mu \right) |(dx) \underset{N\rightarrow \infty}{\longrightarrow} \int_{\mathbb{R}^d}(1+|x|^\ell)\sigma(dx) < \infty $\\
			
			\item  $\int_{\mathbb{R}^d}(1+|x|^{\ell (1-\alpha)})| \sqrt{N}\left(\bar{\nu}_N-\mu \right) |(dx) \underset{N\rightarrow \infty}{\longrightarrow} \int_{\mathbb{R}^d}(1+|x|^{\ell (1-\alpha)})\sigma(dx) < \infty$\\
			
			\item $\left|\bar{\nu}_N-\mu \right| (\mathbb{R}^d) = \dfrac{|\sqrt{N}\left( \bar{\nu}_N-\mu\right)| (\mathbb{R}^d)}{\sqrt{N}}\underset{N\rightarrow \infty}{\longrightarrow} 0$\\
			
			\item  $\int_{\mathbb{R}^d}|y|^\ell |\bar{\nu}_N-\mu|(dy) \leq \int_{\mathbb{R}^d}(1+|y|^\ell) |\bar{\nu}_N-\mu|(dy) = \dfrac{\int_{\mathbb{R}^d}(1+|y|^\ell) |\sqrt{N}(\bar{\nu}_N-\mu)|(dy)}{\sqrt{N}} \underset{N\rightarrow \infty}{\longrightarrow} 0.$
			
		\end{enumerate}
		
		To conclude the proof, we have to study the convergence of $\sqrt{N}\left(U\left(\mu_N \right)- U\left(\bar{\nu}_N\right)  \right) $ in the decomposition (\ref{decomposition}).
		
		\subsection*{Limit distribution of $\sqrt{N}\left(U\left(\mu_N \right)- U\left(\bar{\nu}_N\right)  \right)$}	
		\textbf{FIRST STEP}\quad  
		Let us preliminary recall that by Remark \ref{conv_nuntomu}, Theorem \ref{SLLN} implies $W_\ell\left(\mu_N, \mu\right) \underset{N}{\longrightarrow} 0$ a.s.\\
		Let us now define for $i=1,\cdots,N$ and $s\in \left[ 0,1\right] $

		$$\mu_{N}^{i,s} = \frac{1}{N}\sum_{j=1}^{i-1}\delta_{X_j} +\frac{s}{N}\delta_{X_i}+\frac{\left(1-s \right) }{N}\nu_i+\frac{1}{N} \sum_{j=i+1}^{N}\nu_j.$$
		
		We have $\mu_{N}^{N,1}= \mu_{N} $ and $\mu_{N}^{1,0}=\bar{\nu}_N.$ Moreover since for $i=1,\cdots,N-1$ 
		$$\mu_{N}^{i,1} = \mu_{N}^{i+1,0}, $$
		one has
		
		$$\sqrt{N}\left(U\left(\mu_N \right)- U\left(\bar{\nu}_N\right)  \right) = \sqrt{N} \sum_{i=1}^{N}\left(U(\mu_{N}^{i,1}) -U(\mu_{N}^{i,0})  \right) .$$
		
		Let us now show that   
		$$\max_{1\leq i\leq N} \sup_{s\in \left[0,1 \right]} W_\ell(\mu_{N}^{i,s},\mu) \underset{N\rightarrow \infty}{\longrightarrow} 0 \quad a.s.$$
		
	 Since for  $i=1,\cdots,N$ and $s\in \left[0,1 \right]$
		
		$$ \mu_{N}^{i,s} = s\mu_{N}^{i,1}+(1-s)\mu_{N}^{i-1,1}$$
		under the convention $\mu_{N}^{0,1}=\mu_{N}^{1,0},$ we can proceed as in the first step of the proof of Theorem \ref{main_teo} to prove that 
		
		\begin{equation}\label{unifconv2}
			\max_{1\leq i\leq N} \sup_{s\in \left[0,1 \right]} W_\ell\left( \mu_{N}^{i,s},\mu\right)  = \max_{0\leq i\leq N}W_\ell\left( \mu_{N}^{i,1},\mu \right).  
		\end{equation}

		By the triangle inequality, 
		
		\begin{equation}\label{triangine}
			W_\ell\left( \mu_{N}^{i,1},\mu \right) \leq  W_\ell\left( \mu_{N}^{i,1},\bar{\nu}_N \right) +  W_\ell\left(\bar{\nu}_N,\mu \right) 
		\end{equation}
		where $W_\ell\left(\bar{\nu}_N,\mu \right)$ goes to $0$ as $N$ goes to infinity.\\Therefore it is sufficient to prove that $\max_{0\leq i\leq N} W_\ell\left( \mu_{N}^{i,1},\bar{\nu}_N \right)\underset{N\rightarrow \infty}{\longrightarrow}0$ a.s. which we are now going to do.

		For $i=1,\cdots, N$
		
		$$ \mu_{N}^{i,1} = \frac{i}{N}\mu_{i}+\frac{1}{N}\sum_{j=i+1}^{N}\nu_j. $$
		
		Hence let $\gamma\in \Pi(\mu_{i},\bar{\nu}_i)$ and define
		$ \tilde{\gamma}(dx,dy)=\frac{i}{N} \gamma(dx,dy) +\frac{1}{N}\sum_{j=i+1}^{N}\nu_j(dx)\delta_x(dy):$
		one has
		$ \tilde{\gamma}(dx,\mathbb{R}^d) = \mu_{N}^{i,1}(dx) $ and $ \tilde{\gamma}(\mathbb{R}^d,dy)=\bar{\nu}_N(dy)$. If $\ell>0$ we have
		
		\begin{align*}
			W_\ell^{\ell\vee1}\left( \mu_{N}^{i,1},\bar{\nu}_N\right)&\leq \int_{\mathbb{R}^d\times\mathbb{R}^d}|x-y|^\ell\tilde{\gamma}(dx,dy)=\frac{i}{N}\int_{\mathbb{R}^d\times\mathbb{R}^d}|x-y|^\ell \gamma(dx,dy)+\frac{1}{N}\sum_{j=i+1}^{N}\int_{\mathbb{R}^d\times\mathbb{R}^d}|x-y|^\ell\nu_j(dx)\delta_x(dy)\\
			&= \frac{i}{N}\int_{\mathbb{R}^d\times\mathbb{R}^d}|x-y|^\ell \gamma(dx,dy).
		\end{align*}
		
		Taking the infimum over $\gamma$, one has
		
		$$W_\ell\left( \mu_{N}^{i,1},\bar{\nu}_N\right)\leq \left( \frac{i}{N}\right) ^\frac{1}{\ell\vee1}W_\ell\left( \mu_i,\bar{\nu}_i\right). $$
		Finally let $\alpha\in\left(0,1 \right) $ so that
		\begin{align*}
			\max_{0\leq i\leq N}  W_\ell\left( \mu_{N}^{i,1},\bar{\nu}_N \right) &=  \max_{1\leq i\leq N}  W_\ell\left( \mu_{N}^{i,1},\bar{\nu}_N \right) \leq \max_{1\leq i \leq N}\left( \frac{i}{N}\right) ^\frac{1}{\ell\vee1}W_\ell\left( \mu_i,\bar{\nu}_i\right)\\
			&\leq \max_{1\leq i\leq N}\left( \frac{i}{N}\right) ^\frac{1}{\ell\vee1}W_\ell\left( \mu_i,\mu\right)+\max_{1\leq i\leq N}\left( \frac{i}{N}\right) ^\frac{1}{\ell\vee1}W_\ell\left( \mu,\bar{\nu}_i\right)\\
			&\leq
			\max_{1\leq i\leq \llcorner\alpha N\lrcorner}\left( \frac{i}{N}\right) ^\frac{1}{\ell\vee1}W_\ell\left( \mu_i,\mu\right)+
			\max_{\llcorner\alpha N\lrcorner< i\leq N}\left( \frac{i}{N}\right) ^\frac{1}{\ell\vee1}W_\ell\left( \mu_i,\mu\right)\\
			&+
			\max_{1\leq i\leq \llcorner\alpha N\lrcorner}\left( \frac{i}{N}\right) ^\frac{1}{\ell\vee1}W_\ell\left( \mu,\bar{\nu}_i\right)+\max_{\llcorner\alpha N\lrcorner< i\leq N}\left( \frac{i}{N}\right) ^\frac{1}{\ell\vee1}W_\ell\left( \mu,\bar{\nu}_i\right)\\
			&\leq \alpha^\frac{1}{\ell\vee1} \sup_{ i}W_\ell\left( \mu_i,\mu\right) +  \alpha^\frac{1}{\ell\vee1}\sup_{i} W_\ell\left(\mu,\bar{\nu}_i \right)+
			\max_{\llcorner\alpha N\lrcorner< i\leq N}W_\ell\left( \mu_i,\mu\right)+ \max_{\llcorner\alpha N\lrcorner< i\leq N}W_\ell\left(\mu,\bar{\nu}_i \right).
		\end{align*}

		For fixed $\alpha$, the sum of the two last terms goes to $0$ as $N$ goes to infinity while the sum of the first two terms is arbitrarily small for $\alpha$ small.\
		Then it is possible to conclude that the left-hand side goes to $0$ almost surely.\\
		As explained in the first step of the Theorem \ref{main_teo}, we can adapt the reasoning to the case $\ell=0$.\\
		
		\textbf{SECOND STEP}\quad We can now reconsider the following stopping time 
		
		$$I_N = \min\left\lbrace 1\leq i \leq N : \exists s\in\left[ 0,1\right] :W_\ell\left(\mu^{i,s}_{N},\mu \right)\geq r  \right\rbrace$$
		
		for the filtration $\left(\mathcal{F}_i=\sigma\left(X_1,\cdots,X_i\right)\right)_{i\geq1}.$
		Proceeding exactly as in the second step of the proof of the Theorem \ref{main_teo} and keeping the same notations, we can obtain the following decomposition

		$$\sqrt{N}\left(U\left(\mu_N \right)- U\left(\bar{\nu}_N\right)  \right) = \sqrt{N}\left(U\left(\mu_N \right)- U\left(\bar{\nu}_N\right) - Q_N \right) + \sqrt{N}Q_N $$
		where 
		$$Q_N :=\frac{1}{N}\sum_{i=1}^{N}\left( \frac{\delta U}{\delta m}(\mu_{N}^{i\wedge I_N,0},X_i) -\int_{\mathbb{R}^d} \frac{\delta U}{\delta m}(\mu_{N}^{i\wedge I_N,0},x)\nu_i(dx) \right)$$
		and for $N\geq N^*$, $U\left(\mu_N \right)- U\left(\bar{\nu}_N\right) - Q_N$ coincides with
		$$R_N = \frac{1_{N\geq N^*}}{N}\sum_{i=1}^{N}\int_{0}^{1}ds\int_{\mathbb{R}^d} \left(\frac{\delta U}{\delta m}(\mu_{N}^{i,s},x)- \frac{\delta U}{\delta m}(\mu_{N}^{i,0},x)\right)(\delta_{X_i}-\nu_i)(dx). $$\\
		
		\textbf{THIRD STEP}\quad Let us first prove the convergence in $L^1$ of $\sqrt{N}R_N$ to $0$. By applying the same argument as in the proof of Theorem \ref{main_teo} this is sufficient to deduce that  $\sqrt{N}\left(U\left(\mu_N \right)- U\left(\bar{\nu}_N\right) - Q_N \right)$ goes to $0$ in probability. \\
		By Assumption \textbf{RU\ref{growth_diff2}}, $\exists C < \infty$, $\exists \alpha \in \left(\frac{1}{2},1 \right] $ such that for $N\geq N^*$
		$$ \bigg\rvert\frac{\delta U}{\delta m}(\mu_{N}^{i,s},x)- \frac{\delta U}{\delta m}(\mu_{N}^{i,0},x)\bigg\rvert \leq C\left((1+|x|^\ell)\left\|\mu_{N}^{i,s}-\mu_{N}^{i,0} \right\|_0^\alpha +(1+|x|^{\ell(\alpha-1)})\left(\int_{\mathbb{R}^d}|y|^\ell|\mu_{N}^{i,s}-\mu_{N}^{i,0}|(dy) \right)^\alpha  \right)  $$ 
		
		with
		$$|\mu_{N}^{i,s} -\mu_{N}^{i,0}|(dy) \leq \frac{s}{N}(\delta_{X_i}+\nu_i)(dy)$$
		and
		$$ \left\| \mu_{N}^{i,s}-\mu_{N}^{i,0}\right\|_0  = |\mu_{N}^{i,s} -\mu_{N}^{i,0}|(\mathbb{R}^d)=\frac{s}{N}|\delta_{X_i}-\nu_i|(\mathbb{R}^d)\leq \frac{s}{N}
		(\delta_{X_i}+\nu_i)(\mathbb{R}^d)= \frac{2s}{N}.$$ 
		Substituting the above quantities and using Young's inequality, one obtains
		\begin{align*}
			&\bigg\rvert\frac{\delta U}{\delta m}(\mu_{N}^{i,s},x)- \frac{\delta U}{\delta m}(\mu_{N}^{i,0},x)\bigg\rvert  \leq C\left((1+|x|^\ell)\left( \frac{2s}{N}\right) ^\alpha+(1+|x|^{\ell(1-\alpha)})\left(\int_{\mathbb{R}^d}|y|^\ell(\delta_{X_i}+\nu_i)(dy) \right)^\alpha \left( \frac{s}{N}\right) ^\alpha \right)\\
			&\phantom{=}\leq \frac{2^\alpha C}{N^\alpha}  \left((1+|x|^\ell)+\alpha\left(|X_i|^\ell+\int_{\mathbb{R}^d}|y|^\ell\nu_i(dy) \right)+(1-\alpha)+\alpha \left(|X_i|^\ell+\int_{\mathbb{R}^d}|y|^\ell\nu_i(dy) \right)+|x|^\ell(1-\alpha)\right)\\
			&\phantom{=}\leq \frac{2^\alpha C}{N^\alpha} \left((1+|x|^\ell)+\left(|X_i|^\ell+\int_{\mathbb{R}^d}|y|^\ell\nu_i(dy) \right)+\frac{1}{2}+ \left(|X_i|^\ell+\int_{\mathbb{R}^d}|y|^\ell\nu_i(dy) \right)+\frac{|x|^\ell}{2}\right) \\
			&\phantom{=}= \frac{2^\alpha C}{N^\alpha}  \left(\frac{3}{2}(1+|x|^\ell)+ 2|X_i|^\ell +2\int_{\mathbb{R}^d}|y|^\ell\nu_i(dy) \right).
		\end{align*}
		
		Therefore
		\begin{align*}
			|R_N|
			&\leq \frac{1_{N\geq N^*}}{N}\sum_{i=1}^{N}\int_{0}^{1}ds\int_{\mathbb{R}^d} \bigg\rvert\frac{\delta U}{\delta m}(\mu_{N}^{i,s},x)- \frac{\delta U}{\delta m}(\mu_{N}^{i,0},x)\bigg\rvert(\delta_{X_i}+\nu_i)(dx)\\
			&\leq \frac{2^\alpha C}{N^{\alpha+1}} \sum_{i=1}^{N}\int_{\mathbb{R}^d}\left(\frac{3}{2}(1+|x|^\ell)+ 2|X_i|^\ell +2\int_{\mathbb{R}^d}|y|^\ell\nu_i(dy) \right)(\delta_{X_i}+\nu_i)(dx)\\
			& = \frac{2^\alpha C}{N^{\alpha+1}}\sum_{i=1}^{N}\left( 3 + \frac{11}{2}\int_{\mathbb{R}^d}|y|^\ell\delta_{X_i}(dy) + \frac{11}{2}\int_{\mathbb{R}^d}|y|^\ell\nu_i(dy)\right)\\
			& = \frac{C_1}{N^{\alpha}}+ \frac{C_2}{N^{\alpha}}\int_{\mathbb{R}^d}|y|^\ell\mu_{N}(dy)+\frac{C_3}{N^{\alpha}}\int_{\mathbb{R}^d}|y|^\ell\bar{\nu}_N(dy)
		\end{align*}

		for some positive constants $C_1,C_2$ and $C_3$.\\
		Finally   
		
		$$\mathbb{E}(\sqrt{N}|R_N|) \leq \frac{C_1}{N^{\alpha-\frac{1}{2}}}+\frac{C_2}{N^{\alpha-\frac{1}{2}}}\int_{\mathbb{R}^d}|y|^\ell\bar{\nu}_{N}(dy)+\frac{C_3}{N^{\alpha-\frac{1}{2}}}\int_{\mathbb{R}^d}|y|^\ell\bar{\nu}_N(dy)$$
		
		and since $\alpha > \frac{1}{2}$ and $\lim_{N\to\infty} \int_{\mathbb{R}^d}|y|^\ell\bar{\nu}_{N}(dy)=\int_{\mathbb{R}^d}|y|^\ell\mu(dy) $, we can conclude that the left-hand side goes to $0$ as $N$ goes to infinity and so the third step is concluded.\\ 
		
		\textbf{FOURTH STEP}\quad We are now going to prove the convergence in distribution of $\sqrt{N}Q_N$ to a Gaussian random variable. In this case, we can immediately apply the Central Limit Theorem for martingales while in Theorem \ref{main_teo} we had to introduce first the Poisson equation before applying it to $K_{3,N}$.
		For $1\leq i \leq N$ let
		
		$$Y_{N,i}=\frac{1}{\sqrt{N}}\left( \frac{\delta U}{\delta m}\left(\mu^{i\wedge I_N,0}_{N},X_i \right)-\int_{\mathbb{R}^d}\frac{\delta U}{\delta m}\left(\mu^{i\wedge I_N,0}_{N},x \right) \nu_i(dx)\right).$$\\
		
		Let us start by \textbf{checking that $\mathbb{E}\left(Y_{N,i}|\mathcal{F}_{i-1}\right)=0$. }\\
		
		Since $X_i$ is independent of $\mathcal{F}_{i-1}$ and $\mu^{i\wedge I_N,0}_{N}=\sum_{j=1}^{i-1}1_{\left\lbrace I_N=j\right\rbrace }\mu^{j,0}_{N}+1_{\left\lbrace I_N>i-1\right\rbrace}\mu^{i,0}_{N}$ is $\mathcal{F}_{i-1}$- measurable, applying the Freezing Lemma one has
		
		\begin{align*}
			\mathbb{E}\left(Y_{N,i}|\mathcal{F}_{i-1}\right) &= \mathbb{E}\left(\frac{1}{\sqrt{N}}\left( \frac{\delta U}{\delta m}\left(\mu^{i\wedge I_N,0}_{N},X_i \right)-\int_{\mathbb{R}^d}\frac{\delta U}{\delta m}\left(\mu^{i\wedge I_N,0}_{N},x \right) \nu_i(dx)\right)|\mathcal{F}_{i-1}\right)\\
			&= \frac{1}{\sqrt{N}}\left(\mathbb{E}\left( \frac{\delta U}{\delta m}\left(t,X_i \right) \right)_{t=\mu^{i\wedge I_N,0}_{N}} - \int_{\mathbb{R}^d}\frac{\delta U}{\delta m}\left(\mu^{i\wedge I_N,0}_{N},x \right) \nu_i(dx)\right) = 0.
		\end{align*}
		
		For what concerns the \textbf{square integrability}, let us check that both the components are squared integrable by applying Hypothesis \textbf{RU\ref{growth2}} and the fact that $\nu_{i} \in \mathbb{P}_\ell\left(\mathbb{R}^d \right).$   \\
		\begin{itemize}
			\item 
			$$\mathbb{E}\left( \frac{\delta U}{\delta m}\left(\mu^{i\wedge I_N,0}_{N},X_i \right)^2\right)\leq C^2 \mathbb{E}\left( \left(1+|X_i|^\frac{\ell}{2} \right)^2 \right)\leq 2C^2 \mathbb{E}\left( 1+|X_i|^\ell \right) < \infty$$ 
			\item $$\mathbb{E}\left(\left( \int_{\mathbb{R}^d}\frac{\delta U}{\delta m}\left(\mu^{i\wedge I_N,0}_{N},x \right) \nu_i(dx) \right) ^2\right) \leq \mathbb{E}\left( \int_{\mathbb{R}^d}\frac{\delta U}{\delta m}\left(\mu^{i\wedge I_N,0}_{N},x \right)^2 \nu_i(dx)  \right) \leq 2C^2 \left( 1+\int_{\mathbb{R}^d}|x|^\ell\nu_i(dx)\right)<\infty.$$\\
		\end{itemize}
		Let us now study the \textbf{convergence of
			$\sum_{i=1}^{N}\mathbb{E}\left(Y^2_{N,i}|\mathcal{F}_{i-1} \right):$}\\
		\begin{align*}
			\sum_{i=1}^{N}\mathbb{E}\left(Y^2_{N,i}|\mathcal{F}_{i-1} \right)&= \frac{1}{N}\sum_{i=1}^{N}\mathbb{E}\left(\left( \frac{\delta U}{\delta m}\left(\mu^{i\wedge I_N,0}_{N},X_i \right)-\int_{\mathbb{R}^d}\frac{\delta U}{\delta m}\left(\mu^{i\wedge I_N,0}_{N},x \right) \nu_i(dx)\right)^2 \arrowvert\mathcal{F}_{i-1} \right)\\
			&=  \frac{1}{N}\sum_{i=1}^{N}\left(  \int_{\mathbb{R}^d}  \frac{\delta U}{\delta m}\left(\mu^{i\wedge I_N,0}_{N},x \right)^2\nu_i(dx)-\left( \int_{\mathbb{R}^d}\frac{\delta U}{\delta m}\left(\mu^{i\wedge I_N,0}_{N},x \right)\nu_i(dx) \right)^2\right) .
		\end{align*}

		As in the fourth step of the proof of the Theorem \ref{main_teo}, thanks to Hypothesis \textbf{RU\ref{derivative_conv2}} and thanks to the a.s. convergence of $\max_{1\leq i\leq N} W_\ell(\mu_{N}^{i,0},\mu)$ to $0$, we obtain the existence of a sequence of random variables $\left( \epsilon_N\right)_{N\geq 0} $ converging a.s. to $0$ such that 
		\begin{equation} \label{control2}
			\forall 1\leq i\leq N\qquad \bigg\rvert\frac{\delta U}{\delta m}\left(\mu^{i\wedge I_N,0}_N,x \right)-\frac{\delta U}{\delta m}\left( \mu,x \right)\bigg\rvert\leq\left(1+|x|^\frac{\ell}{2}\right)\epsilon_N.
		\end{equation}
		
		\underline{First Component}\quad Let us first show that 
		$$\lim_{N\to\infty}\frac{1}{N}\sum_{i=1}^{N} \int_{\mathbb{R}^d}  \frac{\delta U}{\delta m}\left(\mu^{i\wedge I_N,0}_{N},x \right)^2\nu_i(dx) -  \int_{\mathbb{R}^d}  \frac{\delta U}{\delta m}\left(\mu,x \right)^2\mu(dx)=0\,\ a.s. $$
		We can rewrite the difference in the following way
		\begin{align*}
			\phantom{==============}&\frac{1}{N}\sum_{i=1}^{N} \int_{\mathbb{R}^d}  \frac{\delta U}{\delta m}\left(\mu^{i\wedge I_N,0}_{N},x \right)^2\nu_i(dx) - \frac{1}{N}\sum_{i=1}^{N} \int_{\mathbb{R}^d}  \frac{\delta U}{\delta m}\left(\mu,x \right)^2\nu_i(dx)\\
			& \phantom{=}+ \frac{1}{N}\sum_{i=1}^{N} \int_{\mathbb{R}^d}  \frac{\delta U}{\delta m}\left(\mu,x \right)^2\nu_i(dx) -  \int_{\mathbb{R}^d}  \frac{\delta U}{\delta m}\left(\mu,x \right)^2\mu(dx)   \\
			& = \frac{1}{N}\sum_{i=1}^{N} \int_{\mathbb{R}^d}  \left(\frac{\delta U}{\delta m}\left(\mu^{i\wedge I_N,0}_{N},x \right)^2\ -  \frac{\delta U}{\delta m}\left(\mu,x \right)^2 \right) \nu_i(dx)+ \int_{\mathbb{R}^d}  \frac{\delta U}{\delta m}\left(\mu,x \right)^2\left(\bar{\nu}_N-\mu \right)(dx).
		\end{align*}
		
		Since  $\lim_{N\to\infty}W_\ell(\bar{\nu}_N,\mu)=0$, by the characterization (\ref{carat_wass}), we can deduce that the second term of the sum goes to $0$ thanks to Assumption \textbf{RU\ref{growth2}} and Assumption \textbf{RU\ref{continuity}}.
		For what concerns the first term for $i=1,\cdots,N$, using Assumption \textbf{RU\ref{growth2}} and (\ref{control2}), one has
		
		\begin{align*}
			\bigg\rvert \frac{\delta U}{\delta m}\left(\mu^{i\wedge I_N,0}_{N},x \right)^2\ -  \frac{\delta U}{\delta m}\left(\mu,x \right)^2\bigg\rvert&\leq \left( \frac{\delta U}{\delta m}\left(\mu^{i\wedge I_N,0}_{N},x \right) -  \frac{\delta U}{\delta m}\left(\mu,x \right)\right)^2+2\bigg\rvert\frac{\delta U}{\delta m}\left(\mu,x \right) \bigg\rvert \bigg\rvert \frac{\delta U}{\delta m}\left(\mu^{i\wedge I_N,0}_{N},x \right)\ -  \frac{\delta U}{\delta m}\left(\mu,x \right)\bigg\rvert\\
			&\leq \left(1+|x|^\frac{\ell}{2}\right)^2\epsilon^2_N+2C\left(1+|x|^\frac{\ell}{2}\right)^2\epsilon_N \leq 2\left(1+|x|^\ell\right)   \epsilon_N \left(\epsilon_N+2C\right).
		\end{align*}

		Therefore
		\begin{align*}
			\bigg\rvert\frac{1}{N}\sum_{i=1}^{N}\int_{\mathbb{R}^d}  \left(\frac{\delta U}{\delta m}\left(\mu^{i\wedge I_N,0}_{N},x \right)^2\ -  \frac{\delta U}{\delta m}\left(\mu,x \right)^2 \right) \nu_i(dx)\bigg\rvert&\leq2 \epsilon_N \left(\epsilon_N+2C\right) \frac{1}{N}\sum_{i=1}^{N}\int_{\mathbb{R}^d}\left(1+|x|^\ell\right)\nu_i(dx)\\
			&= 2 \epsilon_N \left(\epsilon_N+2C\right) \left(1+\int_{\mathbb{R}^d}|x|^\ell\bar{\nu}_N(dx) \right) 
		\end{align*}

		where the right-hand side goes to $0$ a.s. since $\lim_{N\to\infty}\int_{\mathbb{R}^d}|x|^\ell\bar{\nu}_N(dx)= \int_{\mathbb{R}^d}|x|^\ell\mu(dx)$.\\
		
		\underline{Second Component} We are going to prove that 
		$$\lim_{N\to\infty} \frac{1}{N}\sum_{i=1}^{N} \left( \int_{\mathbb{R}^d}\frac{\delta U}{\delta m}\left(\mu^{i\wedge I_N,0}_{N},x \right)\nu_i(dx) \right)^2 -\int_{\mathbb{R}^d \times \mathbb{R}^d} \frac{\delta U}{\delta m} (\mu,x)\frac{\delta U}{\delta m} (\mu,y)\eta(dx,dy)=0 \,\, a.s. $$
		Let us rewrite the difference in the following way
		
		$$\frac{1}{N}\sum_{i=1}^{N}\left(  \left( \int_{\mathbb{R}^d}\frac{\delta U}{\delta m}\left(\mu^{i\wedge I_N,0}_{N},x \right)\nu_i(dx) \right)^2 -  \left( \int_{\mathbb{R}^d}\frac{\delta U}{\delta m}\left(\mu,x \right)\nu_i(dx) \right)^2\right) $$
		$$+\left( \frac{1}{N}\sum_{i=1}^{N}  \int_{\mathbb{R}^d \times \mathbb{R}^d}\frac{\delta U}{\delta m}\left(\mu,x \right)\frac{\delta U}{\delta m}\left(\mu,y\right)\nu_i(dx)\nu_i(dy)  - \int_{\mathbb{R}^d \times \mathbb{R}^d} \frac{\delta U}{\delta m} (\mu,x)\frac{\delta U}{\delta m} (\mu,y)\eta(dx,dy)\right) $$\\
		Since Hypothesis \ref{wess.conv2} holds, again by the characterization (\ref{carat_wass}), we can deduce that the second term of the sum goes to $0$ a.s. thanks to Hypothesis \textbf{RU\ref{growth2}} and Hypothesis \textbf{RU\ref{continuity}}. For what concerns the first term, using Hypothesis \textbf{RU\ref{growth2}} and (\ref{control2}), one has that for $i=1,\cdots,N$ 
		
		\begin{align*}
			&\left|\left(   \int_{\mathbb{R}^d}\frac{\delta U}{\delta m}\left(\mu^{i\wedge I_N,0}_{N},x \right)\nu_i(dx) \right)^2- \left( \int_{\mathbb{R}^d}\frac{\delta U}{\delta m}\left(\mu,x \right)\nu_i(dx) \right)^2 \right|\\
			&\leq \left(    \int_{}\left( \frac{\delta U}{\delta m}\left(\mu^{i\wedge I_N,0}_{N},x \right)-\frac{\delta U}{\delta m}\left(\mu,x \right)\right) \nu_i(dx)   \right)^2 + 2\left|\int_{}\frac{\delta U}{\delta m}\left(\mu,x \right)\nu_i(dx) \right| \left|\int_{}\left( \frac{\delta U}{\delta m}\left(\mu^{i\wedge I_N,0}_{N},x \right)-\frac{\delta U}{\delta m}\left(\mu,x \right)\right) \nu_i(dx) \right|\\
			&\leq     \int_{}\left( \frac{\delta U}{\delta m}\left(\mu^{i\wedge I_N,0}_{N},x \right)-\frac{\delta U}{\delta m}\left(\mu,x \right)\right)^2 \nu_i(dx)    + 2\int_{}\left|\frac{\delta U}{\delta m}\left(\mu,x \right)\right|\nu_i(dx)  \int_{\mathbb{R}^d}\left|\left( \frac{\delta U}{\delta m}\left(\mu^{i\wedge I_N,0}_{N},x \right)-\frac{\delta U}{\delta m}\left(\mu,x \right)\right)\right| \nu_i(dx)\\
			&\leq 2\epsilon_N^2\int_{\mathbb{R}^d} (1+|x|^\ell)\nu_i(dx)+ 2C\epsilon_N \int_{\mathbb{R}^d}\left(1+|x|^\frac{\ell}{2} \right)  \nu_i(dx)       \int_{\mathbb{R}^d}\left(1+|x|^\frac{\ell}{2} \right)\nu_i(dx) \\
			&\leq 2\epsilon_N^2\int_{\mathbb{R}^d} (1+|x|^\ell)\nu_i(dx)+ 4C\epsilon_N \int_{\mathbb{R}^d}\left(1+|x|^\ell \right)  \nu_i(dx) = \left(2\epsilon_N^2+ 4C\epsilon_N \right) \left(1+ \int_{\mathbb{R}^d} |x|^\ell\nu_i(dx)\right).
		\end{align*}

		Therefore 
		\begin{align*}
			&\frac{1}{N}\sum_{i=1}^{N} \left( \int_{\mathbb{R}^d}\frac{\delta U}{\delta m}\left(\mu^{i\wedge I_N,0}_{N},x \right)\nu_i(dx) \right)^2 - \frac{1}{N}\sum_{i=1}^{N} \left( \int_{\mathbb{R}^d}\frac{\delta U}{\delta m}\left(\mu,x \right)\nu_i(dx) \right)^2\\
			&\phantom{==}\leq  \left(2\epsilon_N^2+4C\epsilon_N \right)\left(1+\int_{\mathbb{R}^d}\left|x\right|^\ell\bar{\nu}_N(dx) \right)
		\end{align*}
		
		and the left-hand side goes to $0$ since $\lim_{N\to\infty}\int_{\mathbb{R}^d}|x|^\ell\bar{\nu}_N(dx)= \int_{\mathbb{R}^d}|x|^\ell\mu(dx)$.\\

		To finally apply the Central Limit Theorem for martingales (see Corollary $3.1$ \cite{HallHeyde}) to conclude that $$\sqrt{N}Q_N=\sum_{i=1}^{N}Y_{N,i}\underset{d}{\Longrightarrow} \mathcal{N}\left(0, \int_{\mathbb{R}^d}  \frac{\delta U}{\delta m}\left(\mu,x \right)^2\mu(dx)-\int_{\mathbb{R}^d \times \mathbb{R}^d} \frac{\delta U}{\delta m} (\mu,x)\frac{\delta U}{\delta m} (\mu,y)\eta(dx,dy) \right),$$ it remains just to verify that the \textbf{Lindeberg condition} holds.
		
		We need to check that for $\epsilon>0$, $\sum_{i=1}^{N}\mathbb{E}\left(Y^2_{N,i} 1_{\left\lbrace Y^2_{N,i}> \epsilon \right\rbrace }| \mathcal{F}_{i-1}\right) $ goes to $0$ in probability as $N\rightarrow \infty$.\\
		By Jensen's inequality and Hypothesis  \textbf{RU\ref{growth2}}, one has
		
		\begin{align*}
			NY^2_{N,i} &= \left( \frac{\delta U}{\delta m}\left(\mu^{i\wedge I_N,0}_{N},X_i \right)-\int_{\mathbb{R}^d}\frac{\delta U}{\delta m}\left(\mu^{i\wedge I_N,0}_{N},x \right) \nu_i(dx)\right)^2\\
			&\leq2\frac{\delta U}{\delta m}\left(\mu^{i\wedge I_N,0}_{N},X_i \right)^2+2 \int_{\mathbb{R}^d}\frac{\delta U}{\delta m}\left(\mu^{i\wedge I_N,0}_{N},x \right)^2\nu_i(dx)\\
			&\leq 4C^2\left( 1+ |X_i|^\ell \right)+ 4C^2 \int_{\mathbb{R}^d}\left( 1+ |x|^\ell \right)\nu_i(dx) = 4C^2\left(2+|X_i|^\ell+\int_{\mathbb{R}^d} |x|^\ell\nu_i(dx)  \right).
		\end{align*}

		Therefore, as in the fourth step of the proof of the Theorem \ref{main_teo}, it is enough to check that for each $\epsilon>0$ 
		
		$$2\times1_{\left\lbrace \frac{2}{N} > \epsilon\right\rbrace }+\sum_{i=1}^{N}\mathbb{E}\left(\frac{|X_i|^\ell}{N} 1_{\left\lbrace |X_i|^\ell>N\epsilon \right\rbrace}  \big\arrowvert\mathcal{F}_{i-1}\right)+ \sum_{i=1}^{N}   \frac{\int_{\mathbb{R}^d}|x|^\ell\nu_i(dx)}{N}  1_{\left\lbrace \int_{\mathbb{R}^d}|x|^\ell\nu_i(dx) > N\epsilon\right\rbrace } $$\\
		goes to $0$ as $N$ goes to infinity (a.s.).\\
		It is immediate to prove that the \textbf{first component} goes to $0$ while the \textbf{second component} goes to $0$ by the independence of $X_i$ by $\mathcal{F}_{i-1}$ and (\ref{lindenberg}). For what concerns the \textbf{third component}, we have that $\bar{y}_N:= \frac{1}{N}\sum_{i=1}^{N}\int_{\mathbb{R}^d}|x|^\ell\nu_i(dx) \rightarrow  \int_{\mathbb{R}^d}|x|^\ell\mu(dx).$ Therefore
		
		$$\frac{\int_{\mathbb{R}^d}|x|^\ell\nu_N(dx)}{N} = \bar{y}_N - \frac{N-1}{N}\bar{y}_{N-1}\underset{N\to\infty}{\longrightarrow} 0 $$
		that implies $\forall \epsilon >0$
		$$\int_{\mathbb{R}^d}|x|^\ell\nu_N(dx) 1_{\left\lbrace   \int_{\mathbb{R}^d}|x|^\ell\nu_N(dx) >N \epsilon \right\rbrace} \underset{N\to\infty}{\longrightarrow} 0  $$ 
		and taking the Cesaro mean we deduce that 
		
		$$\lim_{N\to\infty} \frac{1}{N} \sum_{i=1}^{N} \int_{\mathbb{R}^d}|x|^\ell\nu_i(dx)  1_{\left\lbrace \int_{\mathbb{R}^d}|x|^\ell\nu_i(dx) > N\epsilon\right\rbrace } \leq \lim_{N\to\infty} \frac{1}{N} \sum_{i=1}^{N}\int_{\mathbb{R}^d}|x|^\ell\nu_i(dx)1_{\left\lbrace   \int_{\mathbb{R}^d}|x|^\ell\nu_i(dx) >i \epsilon \right\rbrace} =0$$
		
		and so the proof is concluded.
		
	\end{proof}

\end{document}